\documentclass[reqno]{amsart}
\usepackage{nicematrix}
\usepackage{tikz}
\usepackage{xcolor}
\usepackage{amsmath,bm,bbm,amsthm, amssymb}
\usepackage{mathtools}
\usepackage{fullpage}
\usepackage{color}
\usepackage{dsfont}
\usepackage{animate}
\usepackage{etoolbox}
\usepackage{comment}
\usepackage{pgf}
\usetikzlibrary{calc}
\usetikzlibrary{patterns}
\usetikzlibrary{arrows}
\usetikzlibrary{decorations.pathreplacing}
\usepackage[utf8]{inputenc}
\usepackage{pgfplots}

\usepackage{tcolorbox}
\usepackage{adjustbox}
\usepackage[final]{pdfpages}
\usepackage[ruled,vlined,linesnumbered,algosection,resetcount]{algorithm2e}
\usepackage{blkarray}
\usepackage{arydshln}
\usepackage{wrapfig}
\usepackage[colorlinks=true]{hyperref}

\theoremstyle{plain}
\newtheorem{theorem}{Theorem}[section]   
\newtheorem{lemma}[theorem]{Lemma}
\newtheorem{proposition}[theorem]{Proposition}
\newtheorem{corollary}[theorem]{Corollary}

\newtheorem{maintheorem}{Theorem}

\theoremstyle{definition}
\newtheorem{definition}[theorem]{Definition}

\theoremstyle{remark}

\def\R{\mathbb R}

\def\P{\mathbb P}

\def\E{\mathbb E}

\def\eps{\varepsilon}

\def\00{\mathbf 0}
\def\zz{\mathbf z}

\def\bet{\begin{theorem}}
\def\ent{\end{theorem}}
\def\bec{\begin{corollary}}
\def\enc{\end{corollary}}
\def\bep{\begin{proof}}
\def\enp{\end{proof}}
\def\bede{\begin{definition}}
\def\ende{\end{definition}}
\def\f{\frac}

\def\g{\gamma}
\def\la{\lambda}
\def\es{\emptyset}
\def\su{\subseteq}
\def\ms{\mathsf}
\def\co{\colon}

\def\mc{ \mathcal}
\def\ff{\infty}
\def\PP{\mc P}

\def\one{\mathds1}

\def\d{\mathrm d}

\renewcommand\le{\leqslant}
\renewcommand\ge{\geqslant}

\DeclarePairedDelimiter{\norm}{\lVert}{\rVert}

\def\e{\varepsilon}

\def\bel{\begin{lemma}}
\def\enl{\end{lemma}}
\def\im{\item}
\def\been{\begin{enumerate}}
\def\enen{\end{enumerate}}
\def\beit{\begin{itemize}}
\def\enit{\end{itemize}}

\def\sm{\setminus}

\def\k_d{\kappa_d}

\def\s{\sigma}
\def\co{\colon}

\def\k{\kappa}

\def\bepr{\begin{proposition}}
\def\enpr{\end{proposition}}

\usepackage{soul}
\def\Rex{R_{\ms{stab}}}

\def\tN{\textbf{N}}
\def\bN{\mathbb{N}}

\renewcommand{\theenumi}{\Alph{enumi}}
\usetikzlibrary{intersections, calc, angles, quotes}

\DeclareMathOperator*{\argmax}{arg max}

\begin{document}

\title{Central limit theorem for linear eigenvalue statistics of random geometric graphs}

\author{Christian Hirsch}
\author{Kyeongsik Nam}
\author{Moritz Otto}

\address{Department of Mathematics\\ Aarhus University \\ Ny Munkegade, 118, 8000, Aarhus C, Denmark.}
\email{hirsch@math.au.dk}

\address{Department of Mathematical Sciences\\ KAIST \\ South Korea}
\email{ksnam@kaist.ac.kr}

\address{Mathematical Institute
 \\ Leiden University \\ Netherlands
}
\email{m.f.p.otto@math.leidenuniv.nl}

\begin{abstract}Random spatial networks---that is, graphs whose connectivity is governed by geometric proximity---have emerged as fundamental models for systems constrained by an underlying spatial structure. A prototypical example is the random geometric graph, obtained by placing vertices according to a Poisson point process and connecting two vertices whenever their Euclidean distance is less than a certain threshold. Despite their broad applicability, the spectral properties of such spatial models remain far less understood than those of classical random graph models, such as Erd\H{o}s--R\'enyi graphs and Wigner matrices. The main obstacle is the presence of spatial constraints, which induce  highly nontrivial dependencies among edges, placing these models outside the scope of techniques developed for purely combinatorial random graphs.

In this paper, we provide the first rigorous analysis of Gaussian fluctuations for linear eigenvalue statistics of random geometric graphs. Specifically,  we establish central limit theorems for \(\textup{Tr}[\phi(A)]\), where \(A\) is the adjacency matrix and \(\phi\) ranges over a broad class of suitable (possibly non-polynomial) test functions. In the polynomial setting, we moreover obtain a \emph{quantitative} central limit theorem, including an explicit convergence rate to the limiting Gaussian law. We further obtain polynomial-test-function CLTs for other canonical random spatial networks, including \(k\)-nearest neighbor graphs and relative neighborhood graphs. Our results open new avenues for the study of spectral fluctuations in spatially embedded random structures and underscore the delicate interplay between geometry, local dependence, and spectral behavior.

\end{abstract}
\maketitle

\section{Introduction}
\label{sec:intro}

Spectral graph theory deduces structural and dynamic properties of a graph by studying the eigenvalues and eigenvectors of its adjacency matrix \cite{chung}. This methodology finds applications in diverse areas such as network theory, probability, and computer science. For example, the spectrum of the (normalized) adjacency matrix is closely tied to the mixing times of simple random walks on the graph. Spectral analysis also sheds light on key structural features, including community detection and connectivity properties. Furthermore in computer science, spectral information has been instrumental in developing techniques like spectral clustering, which has become a standard tool across a wide range of applications \cite{luxburg}.

The spectrum of adjacency matrices of large random graphs has been a central area of research, with deep connections to random matrix theory. The most extensively investigated model is the  Erd\H{o}s--R\'enyi graph (ERG) $\mc G_{n,p}$, where every edge is included independently with probability $p$. In the dense regime, where the average vertex degree $np$ converges to infinity, the empirical spectral distribution (ESD) of the (normalized) adjacency matrix converges to the Wigner semicircle law. Whereas, when the average vertex degree $np$ remains constant, the limit of ESD is supported on the entire real line and has a dense set of atoms \cite{chayes1986density}.

Beyond the law of large numbers (LLN), understanding the central limit theorem (CLT) for eigenvalues of random graphs plays a fundamental role in capturing the nature of fluctuations and in revealing finer spectral statistics.
It is known that the normalized linear eigenvalue statistics of the adjacency matrix of ERGs $\mc G_{n,p}$ converges to the Gaussian, both in a dense \cite{shch2} (i.e. $np\rightarrow \infty$ and $p\rightarrow 0$) and a constant average degree regime \cite{shch1,cltsparse} (i.e., $np = \text{constant}$).

While the ERG is an essential model in network theory, it is purely combinatorial: edges appear with a fixed probability, independent of each other. However, many real-world networks--such as those in transportation and sensor networks--are inherently spatial. In these networks, nodes are embedded in a geometric space, and edges form based on physical proximity rather than purely probabilistic rules.

This observation has driven significant interest in spatial random networks, particularly the random geometric graph (RGG), also known as the Gilbert graph \cite{penrose,gilbert}. Unlike the ERG, where edge probabilities are independent of each other, the RGG introduces spatial dependence, where edges are formed if the distance between vertices is less than a certain threshold. This dependence complicates the spectral analysis of RGG and limits the direct applicability of results from the ERG. 
Despite these complications, obtaining a thorough understanding of the spectrum of RGG is crucial, both for addressing fundamental theoretical questions and for practical applications in machine learning on graph-structured data. To the best of the authors' knowledge, the current state of research on the eigenvalues of RGG can be summarized as follows:
\been
\im LLN for the ESD of Euclidean matrices was obtained in a thermodynamic regime \cite[Theorem 2]{bordenave}. Precisely, let $A_n:=\{F_n(X_i - X_j)\}_{1\le i, j\le n}$, where $\{X_i\}_{1\le i \le n}$ are i.i.d. uniformly distributed in $[0, 1]^d$ and $F_n(x):= F(n^{1/d}x)$ for some measurable function $F :[0, 1]^d \rightarrow \R $. Note that the special case $F(x) := \one\{|x| \le r\}$ ($r>0$ is a constant) corresponds to the adjacency matrix of RGG in a thermodynamic regime, i.e. the average vertex degree remains constant as $n\rightarrow \infty$. It was proved in \cite[Theorem 2]{bordenave} that the ESD of $A_n$ converges weakly to a certain distribution. In a follow-up work, high-dimensional settings are considered in \cite{bordenave2}.
\im In \cite{avra}, the authors establish a quantitative bound, in Hilbert--Schmidt norm, between the regularized normalized Laplacian of an RGG and a suitable deterministic geometric graph. Their analysis covers both the dense regime, in which the average degree diverges, and the thermodynamic regime, in which it remains uniformly bounded.
\im In addition to the ESD, the spectral edge was studied for the RGG in a dense regime \cite{adhikari}. It was shown that when the average vertex degree is proportional to the number of vertices, the limiting spectral gap of the normalized Laplacian is strictly less than 1. This stands in sharp contrast to the  ERG $\mc G_{n,p}$, where the spectral gap of the normalized Laplacian converges to 1 whenever $ p \gg \log n / n$.  
\enen

Investigating the CLT for linear eigenvalue statistics offers a powerful lens through which to discern global and mesoscopic structural properties of large random networks.
Although LLN-type results have been established for RGG, a critical gap remains: no CLT has yet been demonstrated for their spectral distribution.
In this paper, we fill this gap by proving   the first CLTs for linear eigenvalue statistics of RGG in the sparse regime of constant average degree. This regime captures a nontrivial balance between connectivity and sparsity and is particularly relevant for modeling large-scale spatial networks; see \cite{strictIneq,nestmann}.

~

To summarize, the main contributions of our work are as follows.

\begin{enumerate}

\item We establish, to the best of our knowledge, the first CLT for linear eigenvalue statistics of random geometric graphs that applies to a broad class of \emph{non-polynomial} test functions.

\item For polynomial test functions, we further obtain a \emph{quantitative} CLT, including an explicit convergence rate (in Wasserstein distance) toward the limiting Gaussian law.

\item We also prove polynomial-test-function CLTs for other canonical random spatial networks, including \(k\)-nearest neighbor graphs and relative neighborhood graphs.
 
\end{enumerate}

We briefly outline the proof strategy. We start with polynomial test functions, for which the relevant linear eigenvalue statistics admit a representation in terms of traces of suitable matrix powers. In this regime, the Malliavin--Stein method of \cite{mehler} is particularly well suited: it exploits the underlying geometric structure of the model and, moreover, provides quantitative rates of convergence.

We then extend the CLT to a broad class of smooth (non-polynomial) test functions. This extension is highly nontrivial. Classical arguments developed for Erdős–Rényi graphs—such as martingale methods or Lindeberg-style replacements—depend critically on edge independence and therefore fail in our spatial setting, where geometric correlations are present. To overcome this obstruction, we adapt the Fourier-analytic approach introduced in \cite{shch1} for the Erdős–Rényi case, and implement substantial modifications to accommodate spatial dependencies. The resulting analysis disentangles these correlations and shows that, even in the absence of independence, our modified Fourier-analytic method remains robust enough to establish CLTs for spatial random networks.

\section{Model and main results}
\label{sec:model} 
Let $\PP_n$ be a homogeneous Poisson point process with intensity 1 on the window
$$W_n := [-n^{1/d}/2, n^{1/d}/2].$$
We write $N=N_n:= |\PP_n|$ for the random number of points.
 Let $G(\PP_n)$ denote a graph on the vertex set $\PP_n$,  where edges are included by certain rules. 
A central object associated with $G(\PP_n)$ is the {adjacency matrix} 
\begin{align}
\label{eq:an}
A_n:= \{A(X,Y;\PP_n)\}_{X,Y\in\PP_n}:= \big\{\one\{X, Y\text{ connected by an edge in $G(\PP_n)$}\}\big\}_{X, Y \in \PP_n}.
\end{align}
Note that the size of $A_n$ is also random, as the number of points $N$ in the Poisson point process is random. 
We consider the (un-normalized) empirical measure
$$L_n := \sum_{i =1}^N \delta_{\la_{i}}.$$
We remark that as the spectrum of the adjacency matrices of the considered networks do not depend on the particular labelling of the points in $\PP_n$, the empirical spectral distribution $L_n$ is well-defined.

   We now define our model of random geometric graphs.
For $r>0$, the random geometric graph RGG($r$) is given as follows. An edge is placed between two vertices whenever their Euclidean distance is at most $r$, i.e. for any two points $X$ and $Y$ in the Poisson point process,
\begin{align*}
    \text{$X$ and $Y$ connected} \Leftrightarrow |X-Y| \le r.
\end{align*}
We consider RGG($r$) with a vertex set $\PP_n,$ and let $A_n$ be its adjacency matrix.

The main results of this paper are central limit theorems for the linear eigenvalue statistics 
\begin{align} \label{basic}
  L_n(f) = \int f \d L_n = \textup{Tr} [f(A_n)].
\end{align} 
for a broad class of test functions \(f:\R\to\R\). Our  main theorem establishes a CLT for {general twice weakly differentiable} test functions satisfying a mild weighted Sobolev condition.
This result goes well beyond the polynomial setting and is expected to be appealing in applications
where one needs genuinely non-polynomial observables.
We then  provide \emph{quantitative} normal approximation for   {polynomial} test functions:
we obtain an explicit \(n^{-1/2}\)-rate in Wasserstein distance. 

\subsection{General test functions}
\label{sss:fun}

Throughout the paper, \(\mc N(0,\sigma_f^2)\) denotes a centered normal random variable with variance \(\sigma_f^2\).

%
%
\begin{maintheorem}[CLT for linear eigenvalue statistics; general test functions]
\label{thm:fun}  Fix \(r>0\), and let \(A_n\) denote the adjacency matrix of \(\mathrm{RGG}(r)\) on the vertex set \(\PP_n\).
Let \(f: \mathbb R \rightarrow \mathbb R\) be a twice weakly differentiable function with \(f(0)=0\)
such that for some constant \(c\neq 0\),
\begin{align}\label{condition}
 \int_\R |f(x)|^2 \textup{sech}^2(cx)\, \d x
 +\int_\R |f'(x)|^2 \textup{sech}^2(cx)\, \d x
 +\int_\R |f''(x)|^2 \textup{sech}^2(cx)\, \d x
 < \infty,
\end{align}
where \(\textup{sech} (y) :=  \frac2{e^{y} + e^{-y}}\) denotes the reciprocal of
\(\cosh (y):= \frac{e^y+e^{-y}}{2}\).
Then
the central limit theorem holds: There exists $	\s_f^2 \ge 0$ such that as $n\rightarrow \infty,$
\begin{align}
	\label{eq:fun2}
  \frac{\textup{Tr} [f(A_n)] - \E \textup{Tr} [f(A_n)] }{\sqrt n}  \overset{\textup{d}}{\rightarrow} \mc N (0,\sigma_f^2).
\end{align}
  
\end{maintheorem}

Our central limit theorem is established for test functions belonging to the class \eqref{condition},
defined through the exponentially decaying weight \(\textup{sech}(cx)\).
This choice of weight allows the test functions to exhibit at most subexponential growth at both \(\pm\infty\),
while still ensuring sufficient regularity and integrability to guarantee variance control.
Hence, the admissible class of test functions is remarkably broad---far beyond the space of bounded
or compactly supported functions.

\subsection{Polynomials: Quantitative convergence rates}
\label{sss:pol}

In the polynomial case,  
we obtain explicit bounds on the Wasserstein distance between the (rescaled) spectral measure and the
Gaussian distribution. Recall that the \(L^1\)-Wasserstein distance between (real-valued) integrable random variables
\(X\) and \(Y\) is given by
\[
d_W(X, Y):= \sup_{g \in \mc F}\big|\E[g(X)] - \E[g(Y)]\big|,
\]
where \(\mc F\) denotes the set of all \(1\)-Lipschitz functions from \(\R\) to \(\R\).

%
%
\begin{maintheorem}[CLT for linear eigenvalue statistics; polynomials]
\label{thm:polt1} Fix \(r>0\), and let \(A_n\) denote the adjacency matrix of \(\mathrm{RGG}(r)\) on the vertex set \(\PP_n\).
Then for any polynomial  \(f\), the limit
\[
\sigma_f^2:=\lim_{n\rightarrow \infty }\frac{\textup{Var} (\textup{Tr} [f(A_n)])}{n}
\]
exists. Moreover, if \(\s_f^2 > 0\), then
\[
d_W\Big( \frac{\textup{Tr} [f(A_n)] - \E \textup{Tr} [f(A_n)] }{\sqrt n} , \mc{N}(0, \s_f^2) \Big) \le \frac{C}{\sqrt n },
\]
where \(C \) depends   on  
$d,r$ and \(f\). Finally, \(\s_f^2 > 0\) if all coefficients of \(f\) are non-negative, unless $f\equiv 0$.
 
\end{maintheorem}

Note that our quantitative normal approximation requires the non-degeneracy assumption \(\sigma_f^2>0\).
Nevertheless, the (properly scaled) fluctuations still vanish in probability even in the degenerate case \(\sigma_f^2=0\).
Indeed, by Markov's inequality, for any \(\e>0\),
\[
\limsup_{n\rightarrow \infty}\P\Big(|\textup{Tr} [f(A_n)] - \E[\textup{Tr} [f(A_n)]]| > \e\sqrt{|W_n|}\Big)
\le \frac1{\e^2}\limsup_{n\rightarrow \infty} \frac{\textup{Var}(\textup{Tr} [f(A_n)])}{|W_n|} = 0.
\]

\subsection{Related works: Erd\H{o}s–R\'enyi graphs and random regular graphs} We briefly review the known results regarding the eigenvalues and the corresponding CLT result for other classes of combinatorial models, including Erd\H{o}s–R\'enyi graphs and random regular graphs.
\subsubsection{Erd\H{o}s–R\'enyi graph}
The
Erd\H{o}s-R\'{e}nyi graph $\mc{G}_{n,p}$ is a fundamental model of random graphs, where every edge is included independently of the other edges. Its spectral statistics has been extensively studied so far \cite{MR4515695, eky2, eky1, MR3800840, MR4288336, MR4021251, MR4089498}. When the average degree tends to infinity (i.e. {$p\gg n^{-1}$}), the empirical spectral distribution (ESD) converges weakly to the Wigner's celebrated semicircle law. 

However, this behavior dramatically changes in the constant-average degree regime (i.e.~$p=\frac{d} n$ with $d$ 
 \emph{fixed}), corresponding to the thermodynamic regime considered in our paper. In this regime, the limit of ESD is supported on the entire real line and has a dense set of atoms \cite{chayes1986density}, contrasting sharply with the semicircle law. The qualitative nature of the limiting distribution varies significantly with the average degree $d$. In particular, \cite{bordenave2017mean} established that the limiting law has a continuous part if and only if 
$d > 1$.

Regarding the CLT in the constant average degree regime,  the ESD is known to satisfy a CLT \cite{cltsparse}: Denoting by $A_n$ the adjacency matrix of $\mc{G}_{n,d/n}$, for any  test function $f \in W^{\infty,1}_0(\R)$,
\begin{align*}
  \frac{\textup{Tr} [f(A_n)] - \E \textup{Tr} [f(A_n)] }{\sqrt n}  \rightarrow \mc N(0,\sigma_f^2),
\end{align*}
where $W^{\infty,1}_0(\R)$ denotes the collection of functions $f:\R \rightarrow \R$ such that there is $g\in L^1(\R)$ satisfying
\begin{align*}
	\int_{-\infty}^\infty g(t)\d t=0  \ \text{ and } \  \int_{-\infty}^x g(t)\d t = f(x),\quad   \forall x\in \R.
\end{align*}
In this work, the authors first derived a CLT for higher moments using combinatorial techniques and then extended it to general test functions via a martingale approach.

\subsubsection{Random $d$-regular graph}
The random $d$-regular graph is a uniform probability measure on the collection of graphs (with a fixed number of vertices) with every vertex degree equal to $d$. Analyzing its spectral properties is significantly more challenging than for ERGs due to the dependencies between edges.
It is known that  for any {fixed} $d\ge 3$, the ESD converges to the Kesten-McKay distribution \cite{mckay1981expected}.

However, CLT for the ESD of random 
$d$-regular graphs remains an open problem.
 Notably, there exists a related model, the $d$-uniformly random permutation matrix $A_n$ on $n$ labels, for which a CLT has been established \cite{cltrrg}: For any fixed $d \ge 2,$ there exists $c\in \R$ such that for any reasonable analytic function $f$,
\begin{align}
 \textup{Tr} [f(A_n)] - cn \overset{\textup{d}}{\rightarrow} Y,
\end{align}
where the limiting distribution $Y$ is a \emph{non-Gaussian} infinitely divisible distribution (see \cite[Theorem 35]{cltrrg} for the detailed statement).

\subsection{Organization}

The manuscript is organized as follows. First, in Section \ref{sec3} we provide some preliminaries on Poisson point processes that will be used throughout the manuscript. Then, Sections \ref{section4} and \ref{sec:fun} contain the proofs of our main results. Finally, in Section~\ref{sec:ex}, we further deduce polynomial-test-function CLTs for   \(k\)-nearest neighbor graphs and relative neighborhood graphs.

Throughout  the proofs in this paper,
we will use the same constant, say $C$, whose value might change from line to line.

\subsection{Acknowledgement}
C.H.~is supported by a research grant (VIL69126) from VILLUM FONDEN. K.N. is supported by the National Research Foundation of Korea (RS-2019-NR040050). M.O. is supported by the NWO Gravitation project NETWORKS under grant agreement no. 024.002.003 and a grant
(W253098-1-035) from the Drs. J.R.D. Kuikenga Fonds voor Mathematici.

%
%
\section{Preliminaries on the Poisson point process} \label{sec3}

In this preliminary section, we present fundamental tools for analyzing the Poisson point process.
Let $(X,\mathcal X)$ be a measurable space and $(\Omega, \mc F, \mathbb P)$ be the underlying probability space. Let $\PP$ be a Poisson process in $X$ with intensity measure $\lambda$ which is $\s$-finite. We may regard $\PP $ as a random element
in the space $\textbf{N}$ of integer-valued $\sigma$-finite measures on $X$ equipped with the smallest $\sigma$-field $\mc N$, making the mappings $\mu \mapsto \mu(B)$ measurable for any $ B\in \mathcal X$.

For $p \ge 1,$
we denote $L^p_{\PP}$ to be the space of all random variables $F \in L^p (\mathbb P)$ such that $F = f (\PP)$ $\mathbb P$-a.s.
for some measurable function $f : \textbf{N} \rightarrow  \mathbb R$. Such a function $f$ is uniquely determined almost surely, and is called a representative
of $F$. Also, for $k \in \bN$, $\PP^{(k)}$ denotes the collection of $k$-tuples $(\textup{\textsf{x}}_1,\cdots,\textup{\textsf{x}}_k) $ such that $\textup{\textsf{x}}_1,\cdots,\textup{\textsf{x}}_k \in \PP$ are mutually distinct.

\subsection{Mecke formula}
The Mecke formula provides an elegant way to compute expectations involving Poisson point processes and is widely used in stochastic geometry.
Throughout this section, we assume that $\PP$ is a Poisson point process on $X$ with intensity measure $\lambda$ which is $\s$-finite.

\begin{lemma}[Mecke formula, Theorem 4.1 in \cite{poisBook}] \label{mecke}
For any measurable function 
 $f: X \times \textup{\textbf{N}} \to [0, \infty),$ 
\[
\mathbb{E} \Big[ \sum_{\textup{\textsf{x}} \in \PP} f(\textup{\textsf{x}}, \PP) \Big] = \int_X \mathbb{E} \left[ f(x, \PP + \delta_{ x}) \right]   \d \lambda(x).
\]
\end{lemma}
The following is a  multi-dimensional version of the Mecke formula.
\begin{lemma}[Multi-dimensional Mecke formula, Theorem 4.4 in \cite{poisBook}] \label{multi}
 Let $k \in \bN$. Then, for
any  measurable function 
 $f: X ^k \times \textup{\textbf{N}} \to [0, \infty),$ 
\[
\mathbb{E} \Big[ \sum_{(\textup{\textsf{x}}_1,\cdots,\textup{\textsf{x}}_k) \in \PP^{(k)}} f(\textup{\textsf{x}}_1,\cdots,\textup{\textsf{x}}_k, \PP) \Big] = \int_{X^k} \mathbb{E} \left[ f(x_1,\cdots,x_k, \PP + \delta_{x_1} + \cdots +\delta_{x_k} ) \right] \d \lambda^{\otimes k}(x_1,\cdots,x_k).
\]
\end{lemma}

\subsection{Difference operator}
In the context of point processes, the difference operator is a tool used to study changes in functionals of the process when a point is added.
Let $f : \textbf{N} \rightarrow  \mathbb R$ be a measurable function. Then, for any $x_1,\cdots,x_n \in X$ and $\mathcal Q \in \textbf{N}$, the \emph{$n$-th difference operator} is defined by
\begin{align*}
  D_{x_1,\cdots,x_n}^n f(\mathcal Q) := \sum_{J \su \{1,\cdots,n\}}(-1)^{n-|J|} f \Big( \mathcal Q + \sum_{j\in J}\delta_{x_j}\Big).
 \end{align*}
 Note that $ D_{x_1,\cdots,x_n}^n$ is symmetric in $x_1,\cdots,x_n$, and the map $(x_1,\cdots,x_n, \mathcal Q) \mapsto D_{x_1,\cdots,x_n}^n (\mathcal Q)$ is measurable on $X^n \times \textbf{N}$.
 In particular, we call
 \begin{align}\label{fdo}
     D_x f(\mathcal Q) := f( \mathcal Q + \delta_{x}) - f(\mathcal Q)
 \end{align}
the \emph{first-order difference operator}, and 
$$D_{x, y}^2 f(\mathcal Q):= f(\mathcal Q + \delta_{ x} + \delta_y ) -f (\mathcal Q + \delta_{ x} ) -f(\mathcal Q + \delta_y ) + f (\mathcal Q)$$
the \emph{second-order difference operator}.

Now, let us consider the Poisson point process $\PP$. For $F \in L^2_{\PP}$ with representative $f$, we define $  D_{x_1,\cdots,x_n}^n F :=   D_{x_1,\cdots,x_n}^n  f ( \PP)$. Note that this definition does not depend on the choice of the representative $f$ almost surely, by the multivariate Mecke formula, Lemma \ref{multi}.

\subsection{Poincar\'e inequality}
The 
Poincar\'e inequality provides bounds on the variance of a functional of a Poisson point process in terms of its gradient. It is particularly useful in the analysis of fluctuations and concentration inequalities.  
\begin{lemma} 
[Poincar\'e inequality] \label{poincare}
For any square-integrable random variables $F \in L^2_{\PP}$,
  \begin{align*}
      \textup{Var}(F) \le \int_X \mathbb E [(D_x F)^2]\d\lambda (x).
  \end{align*}
\end{lemma}

\def\Fe{F_{\ms{Eucl}}}
\def\Le{L_{\ms{Eucl},n}}

%
%

\section{Proof of Theorem \ref{thm:polt1}--Polynomials} \label{section4}

Although Theorem~\ref{thm:fun}, i.e. CLT for general test functions, is the main result of the paper, its proof relies on first
establishing a CLT for polynomial test functions.
Accordingly, we begin by proving Theorem~\ref{thm:polt1}, which provides the polynomial CLT (in fact, with a quantitative
Wasserstein rate), and then use the polynomial CLT as an input to derive Theorem~\ref{thm:fun} via an approximation argument
for general test functions.

We express the trace $L_n(f)$ for polynomials $f$ as a function of the underlying Poisson point process $\PP_n$. 
For integers $1\le p\le p'$, define
\begin{align} \label{permute}
  \Sigma_{p',p} := \big\{\pi\co \{1, \dots, p'\} \to \{1, \dots, p\}\text{ surjective}\big\}.
\end{align} 
Then, one can write $L_n(x^m)$, defined in \eqref{basic} with $f(x):=x^m$, as follows:
\begin{align}
L_n(x^m) &= \sum_{p =2}^ m  \f1{p!}\sum_{(X_1,\dots,X_p) \in \PP_n^{(p)}}\sum_{\pi \in \Sigma_{m,p}}\prod_{j =1}^ mA(X_{\pi(j)}, X_{\pi(j + 1)};{\PP_n}) ,
\end{align}
where we use the convention that $\pi(p'+1)  = \pi(1)$. 
More generally, for a polynomial $f(x) = a_mx^m + \cdots + a_0$, {by linearity of $L_n(\cdot)$},  
we have the expression
\begin{align}\label{eq:g}
 L_n(f) =  \sum_{p = 2}^m\f1{p!}\sum_{(X_1,\dots,X_p) \in \PP_n^{(p)}}  \sum_{q = p}^ma_q   \sum_{\pi \in \Sigma_{q,p}}\prod_{j =1}^{q} A(X_{\pi(j)}, X_{\pi(j + 1)} ; \mc \PP_n).
\end{align}

{Now, let us define a score function.  Let $ {\textbf{N}} $ be the set of counting measures on $ \mathbb R^d $ that are simple  and locally finite.  For any $\mc Q \in \textbf N$ and $\textsf{x}_1\in \mc Q$, the \emph{score} function is defined to be   
\begin{align}\label{score}
  g_f(\textsf{x}_1, \mc Q) &:= \sum_{p = 2}^m  {\f1{(p-1)!}}  \sum_{\substack{(\textsf{x}_2,\dots,\textsf{x}_p) \in \mc (Q\setminus \{\textsf{x}_1\} )^{(p-1)}}} \sum_{q = p}^ma_q \sum_{\substack{\pi \in \Sigma_{q,p} \\ \pi(1)=1} }\prod_{j =1}^{q} A(\textsf{x}_{\pi(j)}, \textsf{x}_{\pi(j + 1)} ; \mc Q).
\end{align}
In other words, \(g_f(\mathsf{x}_1, \mathcal{Q})\) quantifies the contribution of those paths in \(\mathcal{Q}\) that originate at the point \(\mathsf{x}_1\).
}
{Then, by \eqref{eq:g}, $  L_n(f)$ can be expressed as a sum of score functions:
\begin{align} \label{sumofscore}
  L_n(f) = \sum_{X  \in \PP_n} g_f(X, \PP_{n}).
\end{align}
}

This score-function representation will play a crucial role in the proof of Theorem~\ref{thm:polt1}.

%
\subsection{Variance asymptotics}
\label{ss:var2}
In this section, we show the variance asymptotics, i.e., for any  polynomial $f$, there exists $\s_f^2 \ge0$ such that 
$$\lim_{n\to\ff}\frac{\textup{Var}(L_n(f)) }{n} = \s_f^2.$$
We aim to apply \cite[Theorem 1.1]{trinh2}.  From now on,  $B(y,r)$ denotes the closed ball of radius $r$ centered at $y$, and $D_f$ denotes the \emph{add-one cost}
\begin{align} \label{addone}
    D_f(\mc Q) := \textup{Tr} \big[f\big(A^{\mc Q \cup \{0\}}\big)\big] - \textup{Tr} \big[f\big(A^{\mc Q}\big)\big],\qquad \forall \mc Q \in \mathbf N.
\end{align} 

\bep[Proof of the variance asymptotics] 
To apply \cite[Theorem 1.1]{trinh2}, 
we verify the following two conditions:
\been
\item  There exists an almost surely finite  random variable $\Rex$ (depending on $f$),  such that 
for any  $\mc A \in \tN $ with $\mc A \su \R^d \sm B(0, \Rex)$,
\begin{align}
	\label{eq:stabex2222}
	D_f\big((\PP \cap B(0, \Rex))\cup \mc A\big)  = D_f\big(\PP \cap B(0, \Rex)\big).
\end{align}
\item  Moment condition: There exists $p > 2$ (depending on $f$)  such that  
\begin{align}
	\label{eq:stabexm2222}
	\sup_{0 \in W:\text{cube}} \E\big[|D_f(\PP\cap W)|^p\big] < \ff.
\end{align} 
\enen
We first claim that 
 the condition \eqref{eq:stabex2222} implies the weakly stabilizing condition in \cite[Theorem 1.1]{trinh2}.
Denoting by  $n_0$ the smallest (random) integer $n$ such that $B(0, \Rex) \su W_n$,   for any $n \ge n_0$,
\begin{align}\label{weak}
   D_f (\PP_n) \overset{\eqref{eq:stabex2222}}{=} D_f\big(\PP_n \cap B(0, \Rex)\big)  = D_f\big(\PP_{n_0} \cap B(0, \Rex)\big), 
\end{align}
thereby satisfying the weakly stabilizing condition in \cite[Theorem 1.1]{trinh2}. 
In addition, the functional $ \textup{Tr} \big[f\big(A^{\mc Q}\big)\big]$ for the random geometric graph is translation-invariant.
Therefore by \cite[Theorem 1.1]{trinh2}, it suffices to verify the conditions \eqref{eq:stabex2222} and \eqref{eq:stabexm2222}.

Assume that   $f$ is a degree-$m$ polynomial.
We first verify that   $\Rex = rm$ satisfies \eqref{eq:stabex2222}. Indeed in a random geometric graph, the presence of edges is monotone with respect to the underlying point set. 
Consequently, when the origin is added, the add-one cost \(D_f(\cdot)\) arises solely from paths that pass through the origin.
 Moreover, any \(m\)-path passing through \(0\) is contained in \(B(0, rm)\); hence changes to the point configuration outside \(B(0, rm)\) do not affect the add-one cost \(D_f(\mathcal{P} \cap B(0, rm))\).

Now, let us verify \eqref{eq:stabexm2222}.  By linearity, it suffices to consider the case $f(x) = x^m$. Note that the number of $m$-paths originating from $0$ and consisting of points in $\PP \cap W$ is bounded above by $\PP(B(0, rm))^m$, uniformly over all cubes $W$. This implies $|D_f(\PP \cap W) |\le \PP(B(0, rm))^m$. Since the Poisson distribution admits finite moments of all orders, $\E\big[\PP\big(B(0, rm)\big)^{3m}\big] < \ff$. Hence, the moment condition \eqref{eq:stabexm2222}   is satisfied with $p = 3$ (in fact, it holds for any $p > 2$; we simply take $p = 3$ for convenience). Therefore we conclude the proof.
\enp

%
%
{
\bep[Proof of the variance positivity] 
Now we show that the limiting variance is strictly positive when all coefficients of the polynomial $f$ are non-negative, unless $f\equiv 0$.
By \cite[Theorem 1.1]{trinh2}, it is enough to show that
$$
\P(D_f\big(\PP_{n_0} \cap B(0, \Rex)\big) \neq 0)>0,
$$
where $n_0$ is as in \eqref{weak}. Recalling the score function $g_f$ in \eqref{score}, by \cite[Lemma 5.2]{mal_stab},  
$$
D_f\big(\PP_{n_0} \cap B(0, \Rex)\big) = g_f(0, (\PP_{n_0}\cup \{0\} )  \cap B(0,\Rex)) + \sum_{X \in \PP_{n_0} \cap B(0,\Rex)} D_0g_f(X,\PP_{n_0}\cap B(0,\Rex))
$$
(recall that  $D_0$ denotes the difference operator, see \eqref{fdo}).
 Note that, since all coefficients of the polynomial $f$ are assumed to be non-negative, we have that a.s.,
$$D_0g_f(X,\PP_{n_0}\cap B(0,\Rex)) \ge 0,\quad  \forall X \in \PP_{n_0}.
$$
Moreover, $g_f(0,(\PP_{n_0}\cup \{0\} ) \cap B(0,\Rex)) >0$ if there exists a point  in $\PP_{n_0}$ within distance $r$ from  the origin. Thus
$$
\P(D_f\big(\PP_{n_0} \cap B(0, \Rex)\big) \neq 0) \ge \P(\PP_{n_0} \cap B(0,r) \neq \es) >0,
$$
which concludes the proof of positivity of the limiting variance.
\enp
}

%
%
\subsection{Quantitative normal approximation}
\label{ss:clt2}
We now turn to the quantitative CLT.  To accomplish this, we utilize the Malliavin-Stein framework of quantitative normal approximation of Poisson functionals \cite[Theorem 1.1]{mehler}: 
For any square-integrable functionals $F \in L^2_{ \PP_n}$  
such that $\E F = 0$ and $\text{Var} F=1$ satisfying
\begin{align*}
	\E\Big[ \int_{ W_n} (D_{x} F)^2 \d  x\Big]< \infty,
\end{align*}
we have 
$
	d_W\big(F, \mc N(0,1) \big) \le 4 \sqrt{\g_1'} + \sqrt{\g_2'} + \g_3',
	$
where $\g_1',\g_2',\g_3'$ denote the error terms defined as
\begin{align}
	\label{eq:g1}
	\g_1' &:=	\int_{ W_n^3} \sqrt{\E[(D_{x}F)^2(D_{y}F)^2]}\sqrt{\E[(D^2_{x, z}F)^2(D^2_{y, z}F)^2]} \d  x \d y \d z,\\
	\label{eq:g2}
	\g_2' &:=	\int_{ W_n^3} {\E[(D^2_{x, z}F)^2(D^2_{y, z}F)^2]} \d  x \d y\d z,\\
	\label{eq:g3}
	\g_3' &:=	\int_{ W_n} {\E[|D_{x}F|^3]} \d  x.
\end{align}
Applying this to the normalized random variable $$\widehat F :=\frac{F(\PP_n) - \E[F(\PP_n)]}{\sqrt{\textup{Var}[F(\PP_n)]}}
$$ with 
\begin{align}\label{deff}
  F(\PP_n) := L_n(f),
\end{align}
we have
\begin{align} \label{clt}
	d_W\big(\widehat F, \mc N(0,1) \big) \le 4\textup{Var}(F)^{-1}\sqrt{\g_1'} + \textup{Var}(F )^{-1}\sqrt{\g_2'} + \textup{Var}(F )^{-3/2}\g_3',
\end{align} 
where $\g_1',\g_2',\g_3'$ are the  error terms defined in \eqref{eq:g1}-\eqref{eq:g3} with $F$ given in \eqref{deff}.

Therefore, it will be crucial to control the moments of the first- and second-order difference operators. This   is established in the following lemma.

\renewcommand{\theenumi}{\roman{enumi}}%
%
%
\bel[Fourth moment bounds for the difference operators]
\label{lem:diffmom2} Consider \textup{RGG($r$)} with any given $r>0$. Let $f$ be a degree-$m$ polynomial and let $F=F(\PP_n)$ be defined in \eqref{deff}. Then, there exists   $C>0$ (depending on $d,r$ and $f$) such that the following holds for all $n \in \bN$ and $ x,y  \in  W_n$:   
\been 
\im  $\E \big[(D_xF) ^4\big] \le C$.
\im  
$\E\big[(D_{x,y}^2 F)^4\big] \le  C         \one \{|x-y|\le 4mr\}.
$
  
	\enen 
\enl

Given Lemma \ref{lem:diffmom2}, we conclude the proof of Theorem \ref{thm:polt1}.
\bep[Proof of Theorem \ref{thm:polt1}]
 
 Let $f$ be a degree-$m$ polynomial and let $F=F(\PP_n)$ be defined in \eqref{deff}. 
 By H\"older's inequality and Lemma \ref{lem:diffmom2},  
 \begin{align*} \label{410}
 \g_1' &\le \int_{ W_n^3} (\E|D_{x} F|^4)^{1/4} (\E|D_{y} F|^4)^{1/4}  (\E|D_{x,z}^2 F|^4)^{1/4} (\E|D_{y,z}^2 F|^4)^{1/4} \d  x \d y \d z \nonumber \\
   & \le C\int_{ W_n^3}     \one \{|x-z|\le 4mr\}       \one \{|y-z|\le 4mr\}    \d  x \d y \d z  \le C(mr)^{2d}n.
 \end{align*} 
Next, we aim to bound $   \g_2'.$ By H\"older's inequality and Lemma \ref{lem:diffmom2}, 
 \begin{align*}
   \g_2'  \le \int_{ W_n^3} (\E|D_{x,z}^2 F|^4)^{1/2} (\E|D_{y,z}^2 F|^4)^{1/2} \d  x \d y \d z  \le \int_{ W_n^3}   \one \{|x-z|\le 4mr\}       \one \{|y-z|\le 4mr\} \d  x \d y \d z \le C(mr)^{2d}n.
 \end{align*} 
 Finally, by Lemma \ref{lem:diffmom2} again, 
 \begin{align*}
   \g_3' \le \int_{ W_n} {\E[|D_{x}F|^4]}^{3/4} \d  x \le   \int_{W_n} C \d x =Cn.
 \end{align*}
Therefore, if $\sigma_f^2> 0$, then applying the above estimates to \eqref{clt},
\begin{align*}
  d_W\big(\widehat F, \mc N(0,1) \big) \le Cn^{-1}\sqrt{n} + Cn^{-1}\sqrt{n} + Cn^{-3/2}n \le Cn^{-1/2}.
\end{align*}

\enp 
As noted in \eqref{sumofscore}, the functional $F(\PP_n)$ can be written as a sum of scores $g_f$ defined in \eqref{score}. Hence, to derive the moment bounds in Lemma \ref{lem:diffmom2}, it is crucial to obtain corresponding moment bounds for the scores $g_f$. 
For $ \mathcal Q \in  {\textbf{N}}$ and $\mathsf z\in  \mathcal Q,$ the difference operator of score functions is similarly defined as 
$$D_{ x} g_f(\mathsf  z, \mathcal Q) := g_f(\mathsf  z,  \mathcal Q + \delta_{ x}) - g_f(\mathsf   z, \mathcal Q)$$
and 
$$D_{ x, y}^2 g_ f(\mathsf  z, \mathcal Q):= g_f(\mathsf  z, \mathcal Q + \delta_{  x} + \delta_{y} ) -g_f ( \mathsf z, \mathcal Q + \delta_{  x} ) -g_f( \mathsf z, \mathcal Q + \delta_{y} ) + g_f (\mathsf  z, \mathcal Q).$$ 
 For the sake of readability, we introduce the following notation: For $\ell\in \bN$ and $ z_1, \dots,  z_\ell \in  W_n$, writing $ \zz_\ell := \{ z_1, \ldots,  z_\ell\}$,
\begin{align}
    \PP_{n,  \zz_\ell} :=   \PP_n\cup  \zz_\ell.
\end{align}
Also for a single element $x\in  W_n$, we write $  \PP_{n, x} :=  \PP_n\cup\{x\}$ and $   \PP_{n,  \zz_\ell,x } :=   \PP_n\cup  \zz_\ell\cup \{x\}.$ 
\bel[Moment bounds]
\label{lem:mom2}
Consider \textup{RGG($r$)} with any given $r>0$.
Let {$\ell,m \in \bN$} and $f$ be a degree-$m$ polynomial. Then, the following holds for any $ x   ,y   \in  W_n$ and $\zz_\ell := \{z_1,\dots,z_\ell\}$ with $z_i   \in  W_n$ $(i=1,\dots,\ell)$:  
\been
\im For any $k \in \mathbb N,$ there exists a constant $C >0$ {(depending on $r,d,k,\ell$ and $f$)} such that  for any $n\in \mathbb N,$
$$\sup_{n \in \bN}\E[|g_f(z_1,  \PP_{n, \zz_\ell} )|^{k}]\le C .$$
\im   For any $n\in \mathbb N,$
$$
  D_{ x} g_f(z_1,  \PP_{n, \zz_\ell}) \ne 0  \Rightarrow  |x - z_1| \le   mr .$$
\im  For any $n\in \mathbb N,$ 
$$
 D_{ x,y}^2 g_f(z_1,  \PP_{n, \zz_\ell}) \ne 0 \Rightarrow |x - z_1| \vee |y - z_1| \le  mr.$$
\enen 
\enl
%
%
\bep 
\emph{Part (i).} 
By linearity, we may assume that $f(x) = x^m$.  
Observe that expanding the expression \( |g_f(z_1, \mathcal{P}_{n,\mathbf{z}_\ell})|^k \) 
yields a sum of products, each corresponding to the concatenation of \(k\) loops of length \(m\), all of which pass through \(z_1\) at least once.   
Thus writing $ X_1:=z_1,$
\begin{align} \label{501}
  \big|g_f(z_1, \PP_{n, \zz_\ell})\big|^k \le \sum_{p = 2}^{km}  {\f1{(p-1)!}}  \sum_{\substack{( X_2,\dots, X_p) \in   \PP_{n, \zz_\ell} ^{(p-1)}}} \sum_{\substack{\pi \in \Sigma_{km,p} \\ \pi(1)=1} }\prod_{j =1}^{km} A( X_{\pi(j)},  X_{\pi(j + 1)} ;  \PP_{n, \zz_\ell}),
\end{align} 
where we set $\pi(km +1) := \pi(1)$.

For $\pi \in \Sigma_{km,p}$, we associate a directed cycle on the graph with 
vertex set $\{1,2,\dots,p\}$, in which each $i=1,2,\dots,km$ contributes 
a directed edge from $\pi(i)$ to $\pi(i+1)$. Since $\pi$ is surjective, this graph is connected when regarded as an undirected graph. 
Let $\mathsf{Tree}(\pi)$ denote the collection of all spanning trees of this 
directed cycle, viewed as an undirected graph.
 Since every entry of the adjacency matrix is either 0 or 1, for any \(\mc T \in \mathsf{Tree}(\pi)\), 
\begin{align} \label{502}
  \prod_{j=1}^{km} A\bigl( X_{\pi(j)},  X_{\pi(j + 1)};  \PP_{n, \zz_\ell} \bigr)
 \le 
\prod_{\{i,j\}\in E(\mc T)} A\bigl( X_i,  X_j;  \PP_{n, \zz_\ell} \bigr),
\end{align}
where \(E(\mc T)\) denotes the collection of (undirected) edges in $\mc T$ and \(\{i,j\}\) represents the (undirected) edge between vertices \(i\) and \(j\).

When employing the Mecke formula, we must proceed carefully because some \(X_i\) belong to the Poisson point process \(\PP_n\), while others belong to \(\zz_\ell\). To resolve this, we partition the vertex set of each spanning tree. Concretely, for fixed $2\le p\le km$ and a tree \(\mc T\) with \(p=s_1+ s_2\)  vertices labeled \(\{1,2,\dots, p\}\), we write 
\[
V(\mc T) = S_1 \sqcup S_2
\]
(here, $\sqcup$ denotes disjoint union), 
where \(|S_1| = s_1\) and \(|S_2| = s_2\). We designate \(S_1\) as the subset corresponding to points in \(\PP_n\) and \(S_2\) as the subset corresponding to points in \(\zz_\ell\). This partition allows us to apply the Mecke formula correctly while keeping track of which points originate from each source.

For $p\in \bN$, let $T_p$ be the set of trees with labeled vertices $\{1,2,\dots,p\}.$ 
Note that there is $C_m>0$ (depending only on $m$) such that $|\Sigma_{km,p} | \le C_m$ for any $2\le p\le km.$ Thus, using \eqref{501} and \eqref{502},  
\begin{align} \label{505}
   |g_f(z_1,  \PP_{n, \zz_\ell})|^k  
  &\le C_m\sum_{p = 2}^{km} \sum_{\mc T \in T_p} \sum_{ \substack{S_1,S_2 \\ \{1,\dots,p\}=S_1 \sqcup S_2 }} \sum_{ \substack{
  \xi: S_2 \rightarrow \{1,\dots,\ell\} 
 \text{ injective} \\
  1\in \xi(S_2)
  }}  \sum_{(Y_\tau)_{\tau\in S_1} \in  \PP_{n}^{(|S_1|)}} {\f1{|S_1|!}} \Big[\prod_{ \substack{\{i, j\}\in E(\mc T) \\ i,j\in S_1}} A(Y_i,Y_j ;  \PP_{n, \zz_\ell}) \nonumber \\
  & \cdot \prod_{ \substack{\{i, j\}\in E(\mc T) \\ i\in S_1, j\in S_2}} A(Y_i, z_{\xi_j} ; \PP_{n, \zz_\ell})  \cdot \prod_{ \substack{\{i, j\}\in E(\mc T) \\ i,j\in S_2}} A(z_{\xi_i} ,z_{\xi_j} ;  \PP_{n, \zz_\ell}) \Big].
\end{align}
Note that the condition \(1 \in \xi(S_2)\) in the fourth summation ensures that the corresponding vertex set contains \(z_1\), since we are considering  paths that pass through \(z_1\).

Let us now take the expectation of above.
For any fixed $S_1,S_2$ with $\{1,2,\dots,p\}=S_1 \sqcup S_2$ and a tree $\mc T\in T_p$ along with $(z_{\xi_j})_{j\in S_2}\in \zz_\ell$ in the above summation, by Mecke formula  (Lemma \ref{multi}), 
\begin{align} \label{506}
&\E\Big[\sum_{(Y_\tau)_{\tau\in S_1} \in  \PP_{n}^{(|S_1|)}} \Big ( \prod_{ \substack{\{i, j\}\in E(\mc T) \\ i,j\in S_1}} A(Y_i,Y_j ;  \PP_{n, \zz_\ell})\cdot \prod_{ \substack{\{i, j\}\in E(\mc T) \\ i\in S_1, j\in S_2}} A(Y_i, z_{\xi_j} ;  \PP_{n, \zz_\ell})  \cdot \prod_{ \substack{\{i, j\}\in E(\mc T) \\ i,j\in S_2}} A(z_{\xi_i} ,z_{\xi_j} ;  \PP_{n, \zz_\ell})  \Big) \Big ] \nonumber \\
&= \int_{W_n^{|S_1|}}   \E\Big[\prod_{ \substack{\{i, j\}\in E(\mc T) \\ i,j\in S_1}} A(y_i,y_j ;  \PP_{n, \zz_\ell} \cup \{y_i,y_j\} ) \cdot \prod_{ \substack{\{i, j\}\in E(\mc T) \\ i\in S_1, j\in S_2}} A(y_i, z_{\xi_j} ;  \PP_{n, \zz_\ell} \cup \{y_i\} )  \nonumber \\ 
&\quad \quad \quad \quad \cdot \prod_{ \substack{\{i, j\}\in E(\mc T) \\ i,j\in S_2}} A(z_{\xi_i} ,z_{\xi_j} ;  \PP_{n, \zz_\ell}) \Big] \prod_{\tau\in S_1} \d y_\tau   .
\end{align} 
  Recalling the rule of connections in random geometic graphs, this quantity is written as 
\begin{align} \label{508}
&   	\int_{W_n^{|S_1|}} \Big[ \prod_{ \substack{\{i, j\}\in E(\mc T)  }}  \one \{|y_i- y_j|\le r\} \cdot \prod_{ \substack{\{i, j\}\in E(\mc T)  }}  \one \{|y_i- z_{\xi_j}|\le r\} \prod_{ \substack{\{i, j\}\in E(\mc T)  }}  \one \{|z_{\xi_i}- z_{\xi_j}|\le r\} \Big] \prod_{\tau\in S_1} \d y_\tau .
\end{align}
To estimate this,  we regard $\mc T$ as a tree rooted at $z_1$, and then integrate out the variables \(\{y_k\}_{k \in S_1}\) one by one, starting from those associated with leaf vertices in   \(\mc T\). Concretely, let \(u\) be a leaf and let \(v\) be the unique vertex connected to \(u\) after the series of procedures. We use the following inequalities:
$$\int_{W_n}  \one \{|y_u- w|\le r\}  \d y_u \le Cr^d \ \text{ and } \  \one \{|z_{\xi_u}- w|\le r\}  \le 1, \qquad \forall w\in W_n.$$  
We repeat this procedure for each leaf (and subsequently for any new leaves that appear after each integration).
 Observe that a tree \(\mc T\)  has exactly \(|\mc T|-1\) edges, and each edge involves at most one \(y\)-variable integration step. Furthermore, since \(z_1\) serves as the root of \(\mc T\), once all integration steps are complete, only \(z_1\) remains. Because \(z_1\) is not an integration variable, it does not contribute any additional factor to the bound.  Therefore,  we conclude the proof.

\medskip

%
%
\emph{Part (ii).} Observe   that $D_{ x} g_f(z_1,  \PP_{n, \zz_\ell}) =   g_f(z_1, \PP_{n, \zz_\ell,  x}) - g_f(z_1, \PP_{n, \zz_\ell}) \ne 0$ implies the 
existence of $z_1= X_1, X_2,\dots,  X_{p-1},X_p =x\in  \PP_{n, \zz_\ell ,  x}$ with $2\le p \le m$ such that $ X_i$ and $X_{i + 1}$ are connected for $i=1,\dots,p-1.$ By triangle inequality, we obtain  
 $|x-z_1| \le  mr.$  
\medskip

%
%
\emph{Part (iii).} 
One can write $D_{x,y}^2g_f(z_1, \PP_{n,\zz_\ell})$ as a difference of $D_{ x}g_f(z_1,  \PP_{n, \zz_\ell ,y})$ and $D_{ x}g_f(z_1,  \PP_{n, \zz_\ell})$. Hence, $D_{x,y}^2g_f(z_1, \PP_{n,\zz_\ell})\ne 0$ implies that either $D_{ x}g_f(z_1,  \PP_{n, \zz_\ell ,y})$ or $D_{ x}g_f(z_1,  \PP_{n, \zz_\ell})$ is non-zero.
Thus by part (ii),
\begin{align*}
 D_{ x, y}^2 g_f(z_1,  \PP_{n, \zz_\ell}) \ne 0 \Rightarrow    |x - z_1|  \le mr  .
\end{align*}
Now, interchanging the roles of $ x$ and $y$ gives
\begin{align*}
  D_{ x, y}^2 g_f(z_1,  \PP_{n, \zz_\ell}) \ne 0 \Rightarrow    |y - z_1|    \le mr  .
\end{align*}
Therefore, we conclude the proof. 
\enp

Finally, we proceed with the proof of Lemma \ref{lem:diffmom2}, inspired by the argument of \cite[Lemma 5.5]{mal_stab}. 
%
%
\bep[Proof of Lemma \ref{lem:diffmom2}] We prove the two parts separately. Recall that we write $  \PP_{n, x} :=  \PP_n\cup\{x\}$.
\medskip

\noindent \emph{Part (i).}
Recalling the definition of difference operator, we get as in \cite[Lemma 5.2]{mal_stab} that 
\begin{align}
	\label{lem:52}
	D_{ x}F( \PP_n) = g_f( x, \PP_{n,  x}) + \sum_{ X \in \PP_n} D_{ x} g_f(  X,  \PP_n).
\end{align}
Thus,
$
	\E \big[|D_{ x}F( \PP_n)|^4 \big]
 \le 16 \E\big[ g_f( x, \PP_{n,  x} )^4\big] + 16 \E\big[\big( \sum_{ X \in \PP_n} D_{ x} g_f(  X,  \PP_n)\big)^4\big].
 $
By Lemma \ref{lem:mom2}(i),  
the first term can be controlled as
\begin{align} \label{451}
  \E\big[ g_f( x, \PP_{n,  x} )^4\big] \le C   .
 \end{align}
To bound the second term, let 
$Z:=\#\big\{  X \in  \PP_n\co D_{ x} g_f( X, \PP_n)\ne 0\big\}$
denote the (random) number of non-zero summands in the sum in \eqref{lem:52}.
Then, by Jensen's inequality applied to the function $u \mapsto u^4$,
\begin{align*}
\Big(\sum_{ X \in  \PP_n} D_{ x} g_f( X, \PP_n)\Big)^4 \le Z^4 \sum_{ X \in \PP_n} Z^{-1} | D_{ x} g_f(  X,  \PP_n) | ^4
= Z^3 \sum_{ X \in \PP_n} | D_{ x} g_f(  X,  \PP_n)|^4.
\end{align*}
By deciding whether points in different sums are identical or distinct, we write
$$
\E\big[ Z^3 \sum_{ X \in \PP_n} |D_{ x} g_f(  X,  \PP_n)|^4\big]=I_1+7I_2+6I_3+I_4,
$$
where  the coefficients arise from the Stirling partition number and
$$
I_\ell:=\E \Big[\sum_{( X_1,\dots,  X_\ell)\in  \PP_n^{(\ell)}} |D_{ x} g_f( X_1, \PP_n)|^4\prod_{j=1}^\ell \one\{D_{ x} g_f( X_j, \PP_n)\ne 0\}\Big], \qquad 1\le \ell \le 4.
$$
For $z_i \in  W_n $ with $i=1,\dots,\ell$, we write  $  \zz_\ell:= (z_1,\dots,z_\ell)$, and let   $K = K(z_1,\dots,z_\ell) \in \{1,\dots,\ell\}$ be the index such that
\begin{align} \label{j}
  |x-z_K| = \max \{|x-z_j| : j=1,\dots,\ell\}.
\end{align}
Then, by the multivariate Mecke formula (Lemma \ref{multi}) and H\"older's inequality,
\begin{align*}
	I_\ell&= \int_{ W_n^\ell} \E \Big[ | D_{ x} g_f(z_1, \PP_{n,  \zz_\ell}) | ^4\prod_{j=1}^\ell \one\{D_{ x} g_f(z_j,  \PP_{n,  \zz_\ell}) \ne 0 \} \Big] \d \zz_\ell\\	
   &\le  \int_{ W_n^\ell} \E \Big[ | D_{ x} g_f(z_1, \PP_{n,  \zz_\ell}) | ^4 \one\{D_{ x} g_f(z_K,  \PP_{n,  \zz_\ell}) \ne 0 \} \Big] \d \zz_\ell\\	
	&\le \int_{ W_n^\ell} \big( {\E |  g_f(z_1, \PP_{n,  \zz_\ell,x})|^4} +  {\E | g_f(z_1, \PP_{n,  \zz_\ell})|^4} \big)    \one \{|x - z_K| \le mr\}  \d \zz_\ell \\
    &\le  C   \int_{W_n^\ell}   \one \{|x - z_K| \le mr\} \d \zz_\ell,
\end{align*}   
where we used  Lemmas \ref{lem:mom2} (i) and (ii).
Recalling the index $K$ defined in \eqref{j}, using a change of variables $w_i:= (x - z_i)/m$ and  writing $ \textbf{w}_\ell: = (w_1,\dots,w_\ell)$,  
 \begin{align} \label{442}
	 \int_{W_n^\ell}   { \one \{|x - z_K| \le mr\} } \d \zz_\ell  
	 &= m^{d\ell} \int_{W_n^\ell}   { \one \{ (\max \{|w_1|,\dots,|w_\ell|\}  \le r\}} \d \textbf{w}_\ell \nonumber \\
   &\le      
	 m^{d\ell}\int_{W_n^\ell}  \one \{ |\textbf{w}_\ell| / \sqrt{\ell} \le r \}  \d \textbf{w}_\ell     \le  C ,
\end{align} 
where we used   $\max \{|w_1|,\dots,|w_\ell|\} \ge |\textbf{w}_\ell|/\sqrt{\ell} $ in the first inequality and absorbed the factor $m^{d\ell}$ into the constant $C$ for $1\le \ell \le 4$. Thus, we have $  I_\ell  
  \le   C  $ for $1\le \ell \le 4$.
Therefore, using this along with \eqref{451}, we conclude the proof.

~

%
\emph{Part (ii).} 
Recalling the definition of second order difference operator, by \cite[Lemma 5.2]{mal_stab},
$$
D_{ x,y}^2 F( \PP_n)=D_{ x}g_f(y, \PP_{n, y} ) + D_{y}g_f( x, \PP_{n,  x} ) + \sum_{ X \in \PP_n} D_{ x,y}^2 g_f(  X,  \PP_n).
$$
Thus 
\begin{align} \label{521}
  2^{-8}\E |D_{ x,y}^2 F( \PP_n)|^4 \le \E |D_{ x}g_f(y, \PP_{n, y}) |^4 + \E |D_{y}g_f( x, \PP_{n,  x}) |^4 + \E\Big[\Big(\sum_{ X \in \PP_n} D_{ x,y}^2 g_f( X, \PP_n)\Big)^4\Big]. 
\end{align}
Let us control the first two terms.  
By Lemmas \ref{lem:mom2} (i) and (ii),
 \begin{align} \label{452}
  \E |D_{ x}g_f  (y, \PP_{n, y}) |^4  &= \E [ |D_{ x}g_f(y, \PP_{n, y}) |^4 \one \{|x - y| \le mr\} ] \nonumber \\
    &\le C\big( {\E |  g_f(z_1, \PP_{n,  y,x})|^4} +  {\E | g_f(z_1, \PP_{n, y})|^4} \big)    \one \{|x - y| \le mr\}  \le C   \one \{|x - y| \le mr\}  .
\end{align}
We have the same bound for $\E |D_{y}g_f( x, \PP_{n,  x}) |^4$ as well.

Next, for the last term in \eqref{521}, we argue analogously to Part (i) and find that
$$
\E\Big[\sum_{ X \in \PP_n} |D_{ x,y}^2 g_f( X, \PP_n)|^4\Big] \le J_1+7J_2+6J_3+J_4,
$$
where
$$
J_\ell:=\E\Big[ \sum_{( X_1,\dots, X_\ell)\in \PP_n^{(\ell)}} |D_{ x,y}^2 g_f( X_1, \PP_n)|^4 \prod_{j=1}^\ell \one\{D_{ x,y}^2 g_f( X_j, \PP_n)\ne 0\}\Big], \qquad 1\le \ell \le 4.
$$ For $z_i \in  W_n $ with $i=1,\dots,\ell$, write  $  \zz_\ell:= (z_1,\dots,z_\ell)$, and  let    $K = K(z_1,\dots,z_\ell) \in \{1,\dots,\ell\}$ be   \begin{align} \label{j1}
K: = \argmax_{1\le j \le \ell} (  |x - z_j| \vee |y - z_j| ),
\end{align}
and define
\begin{align}
  A_1 := \{  \zz_\ell \in W_n^\ell: |x-z_K| \ge |y-z_K|\},\quad   A_2 := \{  \zz_\ell \in W_n^\ell: |x-z_K| < |y-z_K|\}.
\end{align}
 With the aid of Mecke formula, by the same reasoning as in Part (i),
\begin{align} \label{444}
   J_\ell &= \int_{ W_n^\ell} \E \Big[ | D_{ x,y}^2 g_f(z_1, \PP_{n,  \zz_\ell}) | ^4\prod_{j=1}^\ell \one\{D_{ x,y}^2 g_f(z_j,  \PP_{n,  \zz_\ell}) \ne 0 \} \Big] d  \zz_\ell \nonumber \\	
   &\le C \int_{ W_n^\ell} \mathbb E (|g_f(z_1, \PP_{n,  \zz_\ell, x,y})|^4 + |g_f(z_1, \PP_{n,  \zz_\ell, x})|^4+|g_f(z_1, \PP_{n,  \zz_\ell,y})|^4+|g_f(z_1, \PP_{n,  \zz_\ell})|^4) \nonumber \\
	&\qquad \qquad \qquad \cdot \one \{ |x - z_K| \vee |y - z_K| \le mr \}\d \zz_\ell  \nonumber \\
    &\le C \int_{ W_n^\ell}  \one \{ |x - z_K| \vee |y - z_K| \le mr \}\d \zz_\ell \nonumber \\ 
    &\le  C \int_{ A_1}  \one \{ |x - z_K|   \le mr \}\d \zz_\ell+ C \int_{  A_2}  \one \{   |y - z_K| \le mr \}\d \zz_\ell.
 \end{align}  
Observe that the first term  above is bounded as
\begin{align*}
  C  \int_{A_1  }       \one \{|x - z_K| \le mr\}    \d \zz_\ell &\le    C \int_{|x - z_K| \ge |x-y|/2}   { \one \{|x - z_K| \le mr\} } \d \zz_\ell \nonumber \\
   &\le    C m^{d\ell} \int_{|\textbf{w}_\ell| \ge |x-y|/(2m
	)}  \one \big\{ |\textbf{w}_\ell| / \sqrt{\ell} \le r \big\} \d \textbf{w}_\ell  \le C  \one \{ |x-y|  \le 2mr \sqrt{\ell}  \} ,
\end{align*}  
where we absorbed the factor $m^{d\ell}$ into the constant $C$ for $1\le \ell \le 4$. 
Similarly, we get the same bound for the second term in \eqref{444}. Hence,  we deduce $\max_{1\le \ell \le 4}  J_\ell \le C      \one \{ |x-y|  \le 4mr  \}.$
Therefore, by this along with \eqref{452}, we conclude the proof.  

\enp

\section{Proof of Theorem \ref{thm:fun}--General functions}
\label{sec:fun}
In this section, we extend the CLT from polynomial test functions to a substantially broader class. The main tool enabling this extension is a general approximation theorem, originally developed in the context of  ERG  (see \cite[Proposition 4]{shch1}). We present this result in Section~\ref{ss:key}, where we also formulate the key intermediate statement, Theorem~\ref{main prop}, and demonstrate how it leads to the proof of Theorem~\ref{thm:fun}. The proof of Theorem~\ref{main prop} is then provided in Section~\ref{ss:mp}.

\subsection{Key ingredients}
\label{ss:key}
The following result serves as a key ingredient in extending the central limit theorem to a broader class of test functions. We state here a version adapted from \cite[Lemma~7.9]{cltsparse}; see also \cite[Proposition~4]{shch1} for its original formulation.

%
%
\begin{proposition}
	\label{aux prop}
	Let $(a_n)_{n \ge 1}$ be a sequence of real numbers, 	$\mc L$ be a {$\mathbb R$-vector space of real-valued functions} equipped with a norm $\norm{\cdot}_\mc L$, and $(\xi_i^{(n)})_{i=1}^n$ be a triangular array of random variables. For a function $\phi : \mathbb R \rightarrow \mathbb R$ in $\mc L$, set
  \begin{align*}
    Z_n(\phi): = a_n\cdot  \sum_{i=1}^n (\phi (\xi_i^{(n)}) - \mathbb E [ \phi (\xi_i^{(n)}) ]).
  \end{align*}
	Suppose that the following assumptions hold:
  \begin{enumerate}
    \item There is a dense subspace $\widetilde{\mc L} \su \mc L$ and a quadratic form $V: \widetilde{\mc L} \rightarrow [0,\infty) $ such that for any $\phi \in \widetilde{\mc L}$,
	    \begin{align} \label{2}
      Z_n (\phi) \overset{\textup{d}}{\rightarrow} \mc N (0, V(\phi))
    \end{align}
  as $n\rightarrow \infty$.
    \item There exists $C>0$ such that for any $\phi \in \mc L$ and $n \in \mathbb N$,
    \begin{align} 
     \mathbb E [Z_n(\phi)^2 ] \le  C \norm{\phi}_\mc L^2.
    \end{align}
  \end{enumerate}
	Then, $V$ is continuous on $\widetilde{\mc L}$ which can be uniquely continuously extended to $\mc L$, and \eqref{2} holds for all $\phi \in \mc L.$
\end{proposition}

In view of the preceding proposition, we have already established the central limit theorem for polynomial test functions. To extend this result to a broader class of test functions, it remains to verify the second requirement—namely, the variance upper bound for general test functions. The following theorem provides this verification.

%
%
\begin{theorem}\label{main prop}  
For any constant $c \neq  0$,  there exists $C>0$ (depending only on $c$) such that for any  $f\in H^2(\mathbb R)$ with    $f(0)=0$,
  \begin{align}\label{theorem9}
   \frac1 n \textup{Var} \big(\textup{Tr} [f(A_n) e^{cA_n} ]\big )  \le C ( \norm{f}^2_2 + \norm{f''}^2_2).
\end{align}
Here, $ H^2(\mathbb R)$ denotes the Sobolev space $W^{2,2}(\mathbb R) $ and $\norm{\cdot}_2$ denotes the standard $L^2$-norm with respect to the  Lebesgue measure.
\end{theorem}

 Given Proposition \ref{aux prop} and Theorem \ref{main prop}, we conclude the proof of Theorem \ref{thm:fun}.
\begin{proof}[Proof of Theorem \ref{thm:fun}]
We apply Proposition \ref{aux prop}.
For  $c\neq 0$,
define $\mc L_c$ to be a collection of functions $f: \mathbb R \rightarrow \mathbb R$ whose $\norm{\cdot}_{\mc L_c}$-norm, defined below, is finite:
\begin{align}\label{cnorm}
  \norm{f}_{\mc L_c}^2 := \norm {f (x) \textup{sech}(c x)}_2^2 +\norm {f ' (x) \textup{sech}(c x)}_2^2 +\norm {f ''(x) \textup{sech}(c x)}_2^2.
\end{align}
	For   $f \in \mc L_c$, set 
  \begin{align}
    Z_n(f): = \frac1{\sqrt n} (\textup{Tr} [f(A_n)] - \E \textup{Tr} [f(A_n)])=  \frac1{\sqrt n} (L_n(f) - \mathbb E L_n(f) ).
  \end{align}We claim that Proposition~\ref{aux prop} applies, with $a_n =  \frac1{\sqrt n} $,  $\mathcal L = \mathcal L_c$ and $\widetilde{\mathcal L}  = \text{the set of polynomials}$.
First, the set of polynomial functions is dense in $\mc L_c$ with respect to the $\norm{\cdot}_{\mc L_c}$-norm.
Moreover, Theorem~\ref{thm:polt1} states that for any polynomial $p$, the centered statistic $Z_n(p)$ converges in distribution to a Gaussian random variable $\mc N(0,\sigma_p^2)$.
Observe that the mapping $p \mapsto \sigma_p^2$ defines a quadratic functional, since $p \mapsto \frac{1}{n}\textup{Var} (\text{Tr} [p(A_n)])$ is quadratic for every $n$ and $\sigma_p^2$ is defined as its limit as $n\rightarrow \infty$.

Finally, we check the  condition  (ii) in Proposition~\ref{aux prop}. 
Setting $\psi(x):=f(x)\textup{sech}(cx)$, 
\[
\text{Tr}[f(A_n)]
=\text{Tr}[\psi(A_n)\cosh(cA_n)]
=\tfrac12 \text{Tr}[\psi(A_n)e^{cA_n}]
+\tfrac12 \text{Tr}[\psi(A_n)e^{-cA_n}].
\]
Define $X:=\text{Tr}[\psi(A_n)e^{cA_n}]$ and $Y:=\text{Tr}[\psi(A_n)e^{-cA_n}]$.
Then
\[
\frac1n\textup{Var}\big(\text{Tr}[f(A_n)]\big)
=\frac1n\textup{Var}\Big(\frac{X+Y}2\Big)
\le \frac{1}{2n}\big(\textup{Var} X+\textup{Var} Y\big).
\]
Noting that $\psi\in H^2(\mathbb R)$ and  {$\psi(0)=0$ (since $f(0)=0$),}  by Theorem~\ref{main prop} (recall the condition $c\neq 0$),  
\[
\frac1n\textup{Var} X
\le C\big(\|\psi\|_2^2+\|\psi''\|_2^2\big),
\qquad
\frac1n\textup{Var} Y
\le C\big(\|\psi\|_2^2+\|\psi''\|_2^2\big),
\]
so that
\begin{align}\label{eq:varpsi}
\mathbb E [Z_n(f)^2 ]= \textup{Var}  (Z_n(f)) = \frac1n\textup{Var} \big(\text{Tr}[f(A_n)]\big)
\le C\big(\|\psi\|_2^2+\|\psi''\|_2^2\big)\le C\norm{f}^2_{\mc L_c}.
\end{align}
The proof of the last inequality will be  provided below.
Hence, all conditions of Proposition~\ref{aux prop} are satisfied, and the desired conclusion follows.

We now proceed to establish the final inequality in \eqref{eq:varpsi}.
Set $g(x):=\text{sech}(cx)$. 
Since
\[
g'(x)=-c g(x)\tanh(cx), \qquad
g''(x)=c^2 g(x)\big(2\tanh^2(cx)-1\big),
\]
we have $|g'|\le c g$ and $|g''|\le c^2 g$.
Since $\psi=f g$, we have $\psi''=f''g+2f'g'+fg'',$
and therefore
\begin{align}\label{77}
    \|\psi''\|_2
\le \|f''g\|_2 + 2\|f'g'\|_2 + \|fg''\|_2
\le \|f''g\|_2 + 2c \|f'g\|_2 + c^2 \|fg\|_2 \le  C  \|f\|_{\mc L_c}.
\end{align}
As $    \|\psi \|_2 = \|fg\|_2 \le  \|f\|_{\mc L_c}$, we are done.

\end{proof}

From now on, we aim to establish Theorem \ref{main prop}.

\subsection{Proof of Theorem \ref{main prop}}
\label{ss:mp}
  
To prove Theorem \ref{main prop}, we need the following bound on the number of paths.
\begin{lemma}\label{key lemma}
  For $m\in \bN,$ let $L_m$ be the number of paths of length $m$ in \textup{RGG($r$)} on the point process $\PP_n$. Also for any integer $1 \le \ell \le m$ and any point $x$, 
let $L_m^{(\ell,x)}$ denote the number of paths of length $m$ 
in \textup{RGG($r$)} on the point process $\PP_n \cup \{x\}$ whose $\ell$-th vertex is $x$.

    There exists a constant $C>0$ (depending only on $r$ and $d$) such that for any $m\in \bN,$
  \begin{align*}  
\sup_{n \ge 1}\sup_{1\le \ell \le m} \sup_{x \in W_n}\Big(\Big (\frac1 n \mathbb E  L_m\Big ) \vee \mathbb E L_m^{(\ell,x)}\Big)   \le C^m \frac{m^m}{(\log (m/r^d +1))^m}.
  \end{align*}  

\end{lemma} 

\begin{proof}

  Let us prove the bound for $\frac1 n \mathbb E  L_m$ first.
We partition the domain $W_n$ into $n$ closed  boxes $B_1,\dots,B_{n/r^d}$ of side-length $r$ (we assume that $n/r^d$ is an integer), whose interiors are disjoint but which may intersect along their boundaries. Let $Z_i$ be the number of points in $\PP_n$ that lie in $B_i$. Then, we have a deterministic inequality
\begin{align} \label{338}
  L_m \le \sum_{\substack{i_1,\dots,i_m \in \{1,2,\dots,n/r^d\} \\
  B_{i_{j}}\cap B_{i_{j+1}} \ne \es \ \forall j=1,\dots,m-1 }} Z_{i_1}Z_{i_2}\dots Z_{i_m}.
\end{align}
This is because any path $(v_1,\dots,v_m)$ satisfies $d(v_i,v_{i+1}) \le r$,    implying that the boxes which contains $v_j$ and $v_{j+1}$ respectively are either adjacent or identical.
By the arithmetic-geometric mean inequality,
\begin{align*}
  Z_{i_1}Z_{i_2}\dots Z_{i_m} \le \frac1m ( Z_{i_1}^m + \dots + Z_{i_m}^m).
\end{align*} 
Note that the moments of the Poisson distribution $Z$ (with parameter $\lambda$) satisfies
\begin{align}
  \mathbb E [Z^m]\le \Big(\frac m{\log (m/\lambda + 1)}\Big)^m ,\qquad \forall m\in \bN,
\end{align}
Since $Z_1,\dots,Z_n$ have the Poisson distribution with parameter $\lambda=r^d$ and the number of admissible index tuples $(i_1,\dots,i_m)$ in the summation \eqref{338} 
is bounded by $\frac{n}{r^d}\cdot (3^d)^{m-1}$, 
\begin{align*}
    \mathbb E L_m \le \frac{n}{r^d} \cdot (3^d)^{m-1} \Big(\frac m{\log (m/r^d + 1)}\Big)^m.
\end{align*}
Next, the bound for $\mathbb E L_m^{(\ell,x)}$ follows by an argument entirely analogous to the one above. 
The only difference is that the initial vertex of each path is now fixed at $x$. 
Consequently, the number of admissible index tuples $(i_1,\dots,i_m)$ in the summation \eqref{338} 
is bounded by $(3^d)^{m-1}$ rather than $\frac{n}{r^d}\cdot (3^d)^{m-1}$. 
Moreover, since the number of points in $\PP_n \cup \{x\}$ contained in a  box $B_i$ 
is at most $Z_i+1$, noting that  a Poisson distribution $Z$ (with parameter $\lambda $) satisfies 
\begin{align}
  \mathbb{E}\big[(Z+1)^m\big]
   \le 
  \Big(\frac{Cm}{\log(m/\lambda +1)}\Big)^m,
  \qquad \forall  m\in\mathbb{N}.
\end{align}
we obtain the desired result.

\end{proof}

Given Lemma \ref{key lemma}, we establish Theorem 
 \ref{main prop}.
\begin{proof}[Proof of Theorem \ref{main prop}]

Let $N = N_n$ be the number of points in the Poisson point process $\PP_n$. Let us enumerate the points in $\PP_n$ as $X_1,\cdots,X_N.$
The proof consists of the following several steps. 

\textbf{Step 1. Application of Poincar\'e inequality.}
For $x \in W_n$,
let $A^{(x)}_n$ be the adjacency matrix of the graph obtained by adjoining the new point 
$x$ to the Poisson point process $\PP_n$: 
\[
A^{(x)}_n :=
\left(
\begin{array}{c|c}
  0 & \mathbf{a} \\ \hline
  \mathbf{a}^T & \begin{array}{ccc}
      & & \\
      & A _n & \\
      & &
  \end{array}
\end{array}
\right),
\] 
where $  {\textbf{a}} = (  a _1,\dots,  a _{N}) $ is such that $  a_i=1$ if $x$ is connected to $X_i$, and $  a_i=0$ if $x$ is not connected to $X_i$.
In order to match the size of the matrix 
$A$   with that of $A^{(x)}_n$, we enlarge $A$ by adjoining a zero row and column as follows:
\[
  \overline A_n :=
\left(
\begin{array}{c|c}
  0 & \mathbf{0} \\ \hline
  \mathbf{0}^T & \begin{array}{ccc}
      & & \\
      & A_n & \\
      & &
  \end{array}
\end{array}
\right).
\]
Then, by 
Poincar\'e inequality (Lemma \ref{poincare}), for any  function $f: \mathbb R \rightarrow \mathbb R$ {such that $f(0)=0$}, 
\begin{align} \label{320}
  \textup{Var} \Big( \frac1{\sqrt n} \textup{Tr} [f(A_n) e^{cA_n}] \Big )&\le \frac1 n \int_{W_n} \mathbb E \big \vert  \textup{Tr} [f(  A^{(x)}_n )e^{c  A^{(x)}_n}  ] - \textup{Tr}[ f( A_n) e^{c A _n} ]\big \vert ^2\d x \nonumber \\
  &= \frac1 n \int_{W_n} \mathbb E \big \vert \textup{Tr} [f(  A^{(x)}_n) e^{c  A^{(x)}_n}  - f(\overline A_n) e^{c\overline A_n } ]  \big \vert ^2 \d x,
\end{align}
where in the last identity, we used the fact that eigenvalues of $f(\overline A_n) e^{c\overline A_n }$ are those of $f( A_n) e^{c A _n}$ together with $f(0)e^0=0$. Assume for the moment that \(f\in L^{1}(\R)\), and let \(\widehat f\) denote the Fourier transform of \(f\). Then, by the Fourier inversion formula:
\begin{align}
  f(M) = \int_\mathbb R \widehat{f}(\xi) e^{i\xi M}\d \xi, \qquad \forall \text{ symmetric matrix } M,
\end{align}
we have 
\begin{align} \label{321}
  \textup{Tr} [f( A^{(x)}_n) e^{c A^{(x)}_n}  - f(\overline A_n) e^{c\overline A_n} ] =  \int_\R \widehat{f}(\xi) \textup{Tr} [e^{(i \xi + c)A^{(x)}_n} - e^{(i \xi + c) \overline A_n}] \d \xi.
\end{align}
To control the integrand, we use the following Duhamel's formula from \cite[Chapter VIII.8]{rs72}:
Let $(A(t))_t$ be a matrix–valued function of $t \in \mathbb R$ such that each matrix element is smooth in $t$. Then,
\begin{align}
  \frac{d}{dt}e^{A(t)} = \int_0^1 e^{s A(t)}A'(t)e^{(1-s)A(t)}\d s .
\end{align}
Using this for $A_n(t) := t(i\xi + c)A^{(x)}_n + (1 - t) (i\xi +c)\overline A_n$, one can write 
\begin{align} \label{322}
    \textup{Tr} \big[e^{(i\xi+c)A^{(x)}_n}-e^{(i\xi+c)\overline A_n}\big]
&=\int_0^1     \textup{Tr} \Big(\frac{\mathrm d}{\mathrm dt}e^{A_n(t)}\Big) \mathrm dt \nonumber  \\
&=\int_0^1 \int_0^1     \textup{Tr} \big(e^{sA_n(t)}A_n'(t)e^{(1-s)A_n(t)}\big) \mathrm ds \mathrm dt \nonumber  \\
&=\int_0^1 \int_0^1     \textup{Tr} \big(A_n'(t)e^{(1-s)A_n(t)}e^{sA_n(t)}\big) \mathrm ds \mathrm dt \nonumber \\
&=\int_0^1     \textup{Tr} \big(A_n'(t)e^{A_n(t)}\big) \mathrm dt \nonumber  \\
&=(i\xi+c)\int_0^1     \textup{Tr} \Big[(A^{(x)}_n-\overline A_n) 
e^{(i\xi+c)\big((1-t)\overline A_n+tA^{(x)}_n\big)}\Big] \mathrm dt \nonumber \\
 &=(i\xi+c)  \int_0^1  \textup{Tr} [e^{t(i \xi + c)A^{(x)}_n} (A^{(x)}_n-\overline A_n) e^{(1-t)(i \xi + c)\overline A_n}]\d t .
\end{align}
Let $\widetilde{\textbf{a}} = (\widetilde a_1,\dots,\widetilde a_{N+1}) := (0, \textbf{a}) $ be the (first) row vector of $A^{(x)}_n$.
Noting that $A^{(x)}_n$ is symmetric, defining the $(N+1)$-dimensional vector $\textbf{e}_1 := (1,0,\cdots,0)^\top$, we obtain that 
\begin{align*}
  \big\vert \textup{Tr}& [e^{t(i \xi + c)A^{(x)}_n} (A^{(x)}_n-\overline A_n) e^{(1-t)(i \xi + c)\overline A_n}] \big\vert \\
  &= \Big\vert \sum_{j,k=1}^{N+1} ( e^{t(i \xi + c)A^{(x)}_n})_{j1}\widetilde{a}_k (e^{(1-t)(i \xi + c)\overline A_n})_{kj} +\sum_{j,k=1}^{N+1} ( e^{t(i \xi + c)A^{(x)}_n})_{jk}\widetilde{a}_k (e^{(1-t)(i \xi + c)\overline A_n})_{1j}\Big\vert  \\
  &= \Big\vert \sum_{j=1}^{N+1} (e^{t(i \xi + c)A^{(x)}_n} \textbf{e}_1)_j (\widetilde{\textbf{a}} \ e^{(1-t)(i \xi + c)\overline A_n} )_j +\sum_{j=1}^{N+1} (e^{t(i \xi + c)A^{(x)}_n} \widetilde{\textbf{a}}^\top)_j (e^{(1-t)(i \xi + c)\overline A_n}\textbf{e}_1)_j \Big\vert 
 \\
  &\le \norm{e^{t(i \xi + c)A^{(x)}_n} \textbf{e}_1}_2 \norm{\widetilde{\textbf{a}} \ e^{(1-t)(i \xi + c)\overline A_n}  }_2 + \norm{e^{t(i \xi + c)A^{(x)}_n} \widetilde{\textbf{a}}^\top}_2 \norm{e^{(1-t)(i \xi + c)\overline A_n} \textbf{e}_1}_2 \\
   &= \norm{e^{tcA^{(x)}_n} \textbf{e}_1}_2 \norm{\widetilde{\textbf{a}} \ e^{(1-t)c\overline A_n}  }_2 + \norm{e^{tcA^{(x)}_n} \widetilde{\textbf{a}}^\top}_2 \norm{e^{(1-t)c\overline A_n} \textbf{e}_1}_2 \\
	& = \sqrt{\langle e^{2tcA^{(x)}_n} \textbf{e}_1,\textbf{e}_1 \rangle } \sqrt{\langle e^{2(1-t)c\overline A_n}\widetilde{\textbf{a}}^\top,\widetilde{\textbf{a}}^\top\rangle} + \sqrt{\langle e^{2tcA^{(x)}_n} \widetilde{\textbf{a}}^\top,\widetilde{\textbf{a}}^\top \rangle},
\end{align*}
where in the last equality we used the fact that   $\norm{M \textbf{v}}_2 = \langle M^2 \textbf{v},\textbf{v} \rangle^{1/2}$ for any symmetric matrix $M$ and a vector $\textbf{v}$, along with $e^{(1-t)c\overline A_n} \textbf{e}_1=\textbf{e}_1.$ From this point onward, to simplify notation, we suppress the transpose and write \(\widetilde{\textbf{a}}\) in place of \(\widetilde{\textbf{a}}^{\top}\), with a slight abuse of notation.
Plugging the above bound along with \eqref{322} into \eqref{321},
\begin{align} \label{323}
	\mathbb E & \big \vert  \textup{Tr} [f(  A^{(x)}_n  ) e^{c  A^{(x)}_n}  - f(\overline A_n) e^{c\overline A_n} ]\big \vert ^2 \nonumber \\
	& \le C\mathbb E \Big \vert  \int_\R |\widehat{f}(\xi)| |i\xi +c| \d \xi \cdot  \int_0^1 \big( \sqrt{\langle e^{2tcA^{(x)}_n} \textbf{e}_1,\textbf{e}_1 \rangle } \sqrt{\langle e^{2(1-t)c\overline A_n}\widetilde{\textbf{a}},\widetilde{\textbf{a}}\rangle } +  \sqrt{\langle e^{2tcA^{(x)}_n} \widetilde{\textbf{a}},\widetilde{\textbf{a}} \rangle} \  \big) \d t \Big \vert  ^2\nonumber \\
  &\le C \Big( \int_\R |\widehat{f}(\xi)| |i\xi +c| \d \xi \Big )^2 \cdot  \int_0^1 \mathbb E [ \langle e^{2tcA^{(x)}_n} \textbf{e}_1,\textbf{e}_1 \rangle  \langle e^{2(1-t)c\overline A_n}\widetilde{\textbf{a}},\widetilde{\textbf{a}}\rangle + \langle e^{2tcA^{(x)}_n} \widetilde{\textbf{a}},\widetilde{\textbf{a}} \rangle ] \d t.
\end{align}
 By H\"older's inequality, the first term in the above integral (in $t$) is bounded as
\begin{align} \label{333}
	\mathbb E [ \langle e^{2tcA^{(x)}_n} \textbf{e}_1,\textbf{e}_1 \rangle  \langle e^{2(1-t)c\overline A_n}\widetilde{\textbf{a}},\widetilde{\textbf{a}}\rangle ] &\le \sqrt{  \mathbb E \langle e^{2tcA^{(x)}_n} \textbf{e}_1,\textbf{e}_1 \rangle ^2 } \sqrt{  \mathbb E \langle e^{2(1-t)c \overline A_n^{}}\widetilde{\textbf{a}},\widetilde{\textbf{a}}\rangle ^2 } \nonumber \\
	&\le \sqrt{  \mathbb E \langle e^{2tcA^{(x)}_n} \textbf{e}_1,\textbf{e}_1 \rangle ^2 } \sqrt{  \mathbb E \langle e^{2(1-t)c A^{(x)}_n }\widetilde{\textbf{a}},\widetilde{\textbf{a}}\rangle ^2 } ,
\end{align}
where in the last inequality, we used the fact that $A^{(x)}_n \ge \overline A_n$ entrywise, and that all entries of these matrices, as well as those of 
$\widetilde{\textbf{a}}$, are nonnegative.  
\medskip

\noindent\textbf{Step 2. Bound on $\mathbb E\big[ \langle e^{2tcA^{(x)}_n} \textbf{e}_1,\textbf{e}_1 \rangle ^2\big] $ for $t  \in [0,1].$}
We use the Cauchy-Schwarz inequality to deduce that for any  real symmetric matrix $M$ and a vector $\textbf{v},$
\begin{align}\label{basic2}
  \langle M\textbf{v},\textbf{v} \rangle^2 \le \langle M^2\textbf{v},\textbf{v}\rangle \langle \textbf{v},\textbf{v}\rangle.
\end{align}
As a consequence, 
$
	\mathbb E \big[\langle e^{2tcA^{(x)}_n} \textbf{e}_1,\textbf{e}_1 \rangle ^2\big] \le \mathbb E\big[ \langle e^{4tcA^{(x)}_n} \textbf{e}_1,\textbf{e}_1 \rangle\big] .
	$
Note that $\langle e^{4tcA^{(x)}_n} \textbf{e}_1,\textbf{e}_1 \rangle$ is just a left-top entry of the matrix $e^{4tcA^{(x)}_n}.$ Using Taylor's expansion for this entry and then applying Lemma \ref{key lemma}, for $t\in [0,1]$,  
 \begin{align} 
   \mathbb E \langle e^{4tcA^{(x)}_n} \textbf{e}_1,\textbf{e}_1 \rangle &\le 1+ \sum_{m=1}^\infty  \frac1{m!} |4tc|^m \cdot    \mathbb E  L_m^{(1,x)}  \nonumber 
   \le 1+ \sum_{m=1}^\infty  \frac1{m!} |4tc|^m \cdot C^m \frac{m^m}{(\log (m/r^d+1))^m}\le C,
 \end{align}
 where we used Stirling's approximation $m! \approx \sqrt{2\pi m} (\frac m{e})^m $ in the last inequality. Hence, 
 \begin{align} \label{341}
   \mathbb E\big[ \langle e^{2tcA^{(x)}_n} \textbf{e}_1,\textbf{e}_1 \rangle ^2\big] \le C.
 \end{align}

\medskip

\noindent \textbf{Step 3. Bound on $\mathbb E\big[ \langle e^{2t'c A^{(x)}_n}\widetilde{\textbf{a}},\widetilde{\textbf{a}}\rangle^2\big]$ for $t' \in [0,1].$}
By the inequality \eqref{basic2},  
\begin{align} \label{335}
  \mathbb E\big[ \langle e^{2t'c A^{(x)}_n }\widetilde{\textbf{a}},\widetilde{\textbf{a}}\rangle ^2\big] \le  \mathbb E [ \langle e^{4 t'c A^{(x)}_n }\widetilde{\textbf{a}},\widetilde{\textbf{a}}\rangle \langle \widetilde{\textbf{a}},\widetilde{\textbf{a}}\rangle ].
\end{align}
As the entries of $A^{(x)}_n$ are \emph{not} independent because of the nature of spatial networks, one cannot decouple $\widetilde{\textbf{a}}$ from $A^{(x)}_n$ as in the Erd\H{o}s-R\'enyi case. Instead, we directly control the quantity in \eqref{335}, with the aid of Lemma \ref{key lemma}. We write $\widehat A_n^{(x)} = (\widehat a_{ij}) := e^{4 t' c   A_n^{(x)} }$. For a non-negative integer $m $, denoting by $L_{m,ij}^{(x)}$ the number of paths of length $m$ from $i$ to $j$ consisting of the points in $\PP_n \cup \{x\}$ (we set $L_{0,ij}^{(x)} := \delta_{ij}$), we have
\begin{align} \label{336}
   \widehat a_{ij} = \sum_{m=0}^\infty \frac{(4t'c )^m L_{m,ij}^{(x)}}{m!}.
\end{align} 
Now, we bound the RHS of \eqref{335}. Note that 
$
  \langle e^{4t' c A^{(x)}_n }\widetilde{\textbf{a}},\widetilde{\textbf{a}}\rangle  = \sum_{i,j=1}^{N+1} \widehat a_{ij} \widetilde a_{i}\widetilde a_{j}
$ 
and
$  \langle \widetilde{\textbf{a}},\widetilde{\textbf{a}}\rangle = \sum_{k=1}^{N+1} \widetilde{a}_k^2 = \sum_{k=1}^{N+1} \widetilde{a}_k,$   since every $\widetilde{a}_k$ is either 0 or 1.
Thus, the RHS of \eqref{335} is written as 
\begin{align} \label{337}
  \mathbb E \Big[\Big(\sum_{i,j=1}^{N+1} \widehat a_{ij} \widetilde a_{i}\widetilde a_{j} \Big)\Big(\sum_{k=1}^{N+1} \widetilde{a}_k \Big)\Big] = \mathbb E \Big[ \sum_{i,j,k=1}^{N+1} \widehat a_{ij}\widetilde a_{i}\widetilde a_{j} \widetilde{a}_k\Big].
\end{align}
 Note that 
\begin{align*}
  \Big\vert \sum_{i,j,k=1}^{N+1} \widehat a_{ij} \widetilde a_{ i} \widetilde a_{ j} \widetilde a_k   \Big\vert &\overset{\eqref{336}}{\le}  \sum_{m=0}^\infty \frac{|4t'c|^m}{m!} \sum_{i,j,k=1}^{N+1} L_{m,ij}^{(x)} \widetilde a _{i}\widetilde a_{j} \widetilde a_k  
  \le  \sum_{m=0}^\infty \frac{|4t'c|^m }{m!} \sum_{i,j,k=1}^{N+1} \widetilde a_k  \widetilde a_j L_{m,ji}^{(x)} , 
\end{align*}
where we used $L_{m,ij}^{(x)} =L_{m,ji}^{(x)} $ and $\widetilde a_{i} \le 1.$ Observe that $\sum_{i,j,k=1}^{N+1} \widetilde a_k  \widetilde a_j L_{m,ji}^{(x)} $ counts the number of paths of length $m+2$ whose vertices belong to 
$\PP_n \cup \{x\}$ and whose second vertex is $x$. 
This is because $\widetilde a_k=\widetilde a_j=1$ implies the existence of a 2-path passing through $x$.
 Thus, by Lemma \ref{key lemma},
\begin{align*}
 \mathbb E \Big[ \sum_{i,j,k=1}^{N+1} \widetilde a_k  \widetilde a_j L_{m,ji}^{(x)} \Big] \le C^{m+2} \frac{(m+2)^{m+2}}{(\log (m/r^d+3))^{m+2}}.
\end{align*}
Hence, by the above computations, for $t' \in [0,1],$ one can control \eqref{335} as
\begin{align} \label{344}
	\mathbb E \big[\langle e^{2t'c A^{(x)}_n }\widetilde{\textbf{a}},\widetilde{\textbf{a}}\rangle ^2\big]&\le  	\mathbb E  \Big[\sum_{m=0}^\infty \frac{|4t'c|^m }{m!} \sum_{i,j,k=1}^{N+1} \widetilde a_k   \widetilde  a_j L_{m,ji}^{(x)} \Big] 
    \le   \sum_{m=0}^\infty \frac{|4t'c |^m }{m!} \cdot C^{m+2} \frac{(m+2)^{m+2}}{(\log (m/r^d+3))^{m+2}} \le C.
\end{align}
\medskip

\noindent\textbf{Step 4. Derivation of \eqref{theorem9} under the additional condition $f\in L^1$.}
Plugging \eqref{341} and \eqref{344} into \eqref{333} and \eqref{323}, gives that for any  $x \in W_n$,
\begin{align*}
 \mathbb E \big[ |\textup{Tr} [f(  A^{(x)}_n e^{c  A^{(x)}_n}  - f(\overline A_n) e^{c\overline A_n} ]|^2 \big] \le C  \Big( \int_\R |\widehat{f}(\xi)| |i\xi +c| \d \xi \Big )^2 \le C  \Big( \int_\R |\widehat{f}(\xi)| (|\xi| + |c|) \d \xi \Big )^2 ,
\end{align*}
where the constant $C>0$ is independent of $x \in W_n$.
Applying this to \eqref{320}, noting that $|W_n|=n,$ 
\begin{align*}
   \text{Var} \Big( \frac1{\sqrt n} \textup{Tr} [f( A_n) e^{c A_n} ]\Big ) \le C \Big( \int_\R |\widehat{f}(\xi)| (|\xi| + |c|) \d \xi \Big )^2 & \le C \Big ( \int_\R | \widehat{f}(\xi) |^2 (|\xi| + |c|)^4 \d \xi  \Big)\Big( \int_\R (|\xi| + |c|)^{-2} \d \xi \Big) \\
   &\le C ( \norm{f}^2_2 + \norm{f''}^2_2).
\end{align*}
Here in the last inequality, we used the condition $c \neq 0$ so that the last integral is finite.

~

\noindent\textbf{Step 5. Removal of the condition $f\in L^1$.} 
Define the events 
\begin{align*}
    \mathcal E_n(0):= \{N_n \le n\}, \quad \mathcal E_n(\ell):= \{ e^{\ell-1} n < N_n \le e^{\ell}n\}\quad \text{for $\ell \in \mathbb N.$}
\end{align*}
Then, by a tail bound for the Poisson distribution, for large enough $n$,
\begin{align}\label{ptail}
    \mathbb P (\mathcal E_n(\ell )  ) \le \frac{(en)^{e^{\ell-1}n} e^{-n}}{(e^{\ell-1}n)^{e^{\ell-1}n}} = e^{-(\ell-2) e^{\ell-1} n - n} \quad \text{for $\ell \in \mathbb N.$}
\end{align}
Let $f\in H^2(\mathbb R)$ with $f(0)=0$ be arbitrary. By Sobolev embedding theorem, $\sup_{x \in \mathbb R} |f(x)| \le C_0$ for some $C_0<\infty.$ Take an approximation    $\{f_k\}_{k \ge 1}$ in $  C_c^\infty (\mathbb R) $ with $f_k(0)=0$ and $ \sup_{x \in \mathbb R} | f_k(x)|\le 2C_0$  such that 
\begin{equation}\label{eq:H2+unif}
f_k \to f \text{ in }H^2(\mathbb R) \qquad \text{as $k\rightarrow \infty.$}
\end{equation}Such a sequence  is constructed by first obtaining a suitable $C_c^\infty$ approximation $g_k$ (e.g., through mollification and cutoff), and then applying a correction, $f_k(x): = g_k(x) - g_k(0)\psi(x)$, where $\psi$ is a $C_c^\infty$ function satisfying $\psi(0)=1$.
By the spectral theorem,
\begin{align} \label{612}
 \big\vert \mathbb E  \big( \text{Tr} [f_k(A_n)   e^{cA_n}]\big) ^2  &- \mathbb E \big( \text{Tr} [f (A_n)   e^{cA_n}]\big) ^2 \big\vert  \nonumber \\
&= \Big\vert \mathbb E \Big( \sum_{i=1}^{N_n}  f_k(\lambda_i) e^{c\lambda_i}\Big)^2  - \mathbb E \Big( \sum_{i=1}^{N_n}  f (\lambda_i) e^{c\lambda_i}\Big)^2 \Big\vert  \nonumber  \\
&\le \sum_{\ell=0}^{\infty}\mathbb E  \Big [ \Big\vert  \Big( \sum_{i=1}^{N_n}  f_k(\lambda_i) e^{c\lambda_i}\Big)^2    - \Big( \sum_{i=1}^{N_n}  f (\lambda_i) e^{c\lambda_i}\Big)^2  \Big\vert  ;  \mathcal E_n(\ell) \Big]. 
\end{align}   
Since every eigenvalue of $A_n$  is bounded by the maximum degree,    $\sup_{1\le i\le N_n} \lambda_i \le N_n \le e^\ell n$    under the event $\mathcal E_n(\ell ).$ Hence, the above quantity is bounded by
\begin{align*}
   \sum_{\ell=0}^{\infty}  \mathbb E \Big[\sum_{i=1}^{N_n}  e^{2c\lambda_i} |f_k(\lambda_i)  - f
(\lambda_i) | &|f_k(\lambda_i) +f(\lambda_i) | ; \mathcal E_n(\ell )   \Big]  \\
&\le  4C_0     \sum_{\ell=0}^{\infty}   \mathbb E \Big [ e^\ell n \cdot e^{2ce^\ell n }  |f_k(\lambda_i)  - f
(\lambda_i) | ; \mathcal E_n(\ell ) \Big ] \\
&\overset{\eqref{ptail}}{\le }  4C_0    \sup_{x\in \mathbb R} |f_k(x)  - f
(x) |\cdot \Big[ ne^{2cn} +  \sum_{\ell=1}^{\infty}  e^\ell n \cdot e^{2ce^\ell n }  \cdot    e^{-(\ell-2) e^{\ell-1} n - n}  \Big] \\
&\le C(n)\cdot  C_0    \sup_{x\in \mathbb R} |f_k(x)  - f
(x) |
\end{align*}
for some  $C(n)>0$. Given (fixed) $n$,
the above term converges to 0 as  $k\rightarrow \infty$, since the $H^2 (\mathbb R)$ convergence \eqref{eq:H2+unif} particularly implies a uniform convergence on $\mathbb R.$
Hence,  we deduce that for any (fixed) $n$,
as $k\rightarrow \infty,$
\begin{align*}
\mathbb E   \big( \text{Tr} [f_k(A_n)   e^{cA_n}]\big) ^2   \rightarrow  \mathbb E \big( \text{Tr} [f (A_n)   e^{cA_n}]\big) ^2 .
\end{align*}
By the same reasoning, as $k\rightarrow \infty,$
\begin{align*}
\mathbb E \text{Tr}[f_k(A_n)e^{cA_n}] \rightarrow   \mathbb E \text{Tr}[f(A_n)e^{cA_n}].
\end{align*}
Therefore, for any (fixed) $n$,  as $k\rightarrow \infty,$
\begin{equation}\label{eq:var-conv}
\text{Var} \big(\text{Tr}[f_k(A_n)e^{cA_n}]\big) \rightarrow  \text{Var}  \big(\text{Tr}[f(A_n)e^{cA_n}]\big).
\end{equation}
Applying the bound \eqref{theorem9} to each $f_k$ (which lies in $C_c^\infty \subseteq L^1\cap H^2$ and satisfies $f_k(0)=0$),  
\[
\frac1n  \text{Var}  \big(\text{Tr}[f_k(A_n)e^{cA_n}]\big)
 \le  C \big(\|f_k\|_2^2 +\|f_k''\|_2^2\big).
\]
By \eqref{eq:var-conv}, sending $k\to\infty$, we deduce that
the desired bound holds for any $f\in H^2$ with $f(0)=0$, completing the proof.
  
\end{proof}

%
%
\section{Other types of random spatial networks}
\label{sec:ex}

In this section, we present further examples of random spatial networks for which CLT for the spectral measure can be established. 
We consider the stabilizing networks, which encompass a broad family of network models that are of interest in computational geometry such as \(k\)-nearest neighbor graphs and relative neighborhood graphs (see \cite{stab2} for a comprehensive overview). 
The defining feature of stabilizing networks is that their local structure is determined by the spatial arrangement of points and their geometric relationships, rather than simple distance-based edge rules as in RGGs. Central to their analysis is the notion of a  \emph{stabilization radius}: a random, point-dependent radius beyond which changes in the configuration do not affect the local graph structure (see \cite{mal_stab,stab1,stab11,bern,stab2} for the examples).

We now give the precise definition of the   stabilizing networks.  
Throughout the paper, let $\textbf{N} $ be the set of counting measures on $\mathbb R^d$ that are simple  and locally finite.  
The {\emph{graph neighborhood}, a (possibly empty) set of points connected to a point by an edge, is a measurable map 	$\mathsf N : \mathbb R^d \times \textbf N \to \textbf N $ with the property that 
for any $\mc Q \in \textbf  N $, $\textsf{x,y}\in \mc Q$ and $z\in \mathbb R^d,$ 
\begin{enumerate}
  \item  $\mathsf N (\textsf x, \mc Q)\su \mc Q$;
  \item  $\textsf y\in \mathsf N (\textsf x, \mc Q) $ if and only if $\textsf x\in \mathsf N (\textsf y, \mc Q) $;
  \item Translation invariance: $\mathsf N (\textsf x + z , \mc Q + z) = \mathsf N (\textsf x, \mc Q) +  z$. 
\end{enumerate} 
{Then, for $\mc Q \in \textbf N $, the associated adjacency matrix is defined as $A^{\mc Q} =\{A(\textsf x,\textsf y;\mc Q)\}_{\textsf x,\textsf y\in \mc Q}$. That is, $A(\textsf x,\textsf y;\mc Q)= \one\{\textsf y\in \mathsf N (\textsf x, \mc Q)\}$. In particular, we set $A_n := A^{\PP_n}$.

Stabilization, informally speaking, requires that the impact of adding a point to the network be confined to a (random but bounded) local neighborhood of that point \cite{yukCLT}. We recall the conditions \eqref{eq:stabex2222} and \eqref{eq:stabexm2222}:
 Let $f: \mathbb R \rightarrow \mathbb R$ be a function. 
 
\been
\item  There exists an almost surely finite  random variable $\Rex$ (depending on $f$),  such that 
for any  $\mc A \in \tN $ with $\mc A \su \R^d \sm B\big(0, \Rex\big)$,
\begin{align}
	\label{eq:stabex}
	D_f\big((\PP \cap B(0, \Rex))\cup \mc A\big)  = D_f\big(\PP \cap B(0, \Rex)\big).
\end{align}
Recall that  $D_f$ is defined in \eqref{addone}.
\item  Moment condition: There exists $p > 2$ (depending on $f$)  such that  
\begin{align}
	\label{eq:stabexm}
	\sup_{0 \in W:\text{cube}} \E\big[|D_f(\PP\cap W)|^p\big] < \ff.
\end{align} 
\enen

As a  consequence of  \cite[Theorem 1.1]{trinh2}, the spectral measure of such stabilizing networks satisfy the CLT.

\begin{proposition} 
 Let \(f:\mathbb R\to\mathbb R\) be a   function. Consider a stabilizing network for which the conditions \eqref{eq:stabex} and \eqref{eq:stabexm} are satisfied, and let \(A_n\) denote its adjacency matrix with underlying point process \(\PP_n\).  Then there exists $\sigma_f^2 \ge 0$ such that as $n \rightarrow \infty$, \begin{align}
  \frac{\textup{Tr} [f(A_n)] - \E \textup{Tr} [f(A_n)] }{\sqrt n}  \overset{\textup{d}}{\rightarrow} \mc N (0,\sigma_f^2).
\end{align}
 
\end{proposition}

\begin{proof}
    
 This immediately follows from  \cite[Theorem 1.1]{trinh2}. Indeed, the conditions \eqref{eq:stabex} (which implies weakly stabilizing condition, see \eqref{weak}) and  \eqref{eq:stabexm} satisfy the conditions in  \cite[Theorem 1.1]{trinh2}. {Note that the translation invariance condition is ensured by our assumptions on the neighborhood.}
\end{proof}

In Sections~\ref{sss:kNN} and~\ref{sss:rng}, we establish CLT for two canonical examples of stabilizing networks: the $k$-nearest neighbor graph   and the relative neighborhood graph.
Finally, in Section~\ref{ssec:out} we outline several additional spatial random network models in the literature, and provide a preliminary assessment of both (i) the main obstacles to applying our methods and (ii) the anticipated level of effort required to overcome them. This last discussion   serves as a natural roadmap for future extensions.

\subsection{$k$-nearest neighbor graphs}
\label{sss:kNN}
For a positive integer $k$, the $k$-nearest neighbor
graph  (kNN) is constructed by connecting two distinct points $X_i$ and $ X_j$ whenever $X_i$ is one
of the $k$-nearest neighbors of $X_j$ or $X_j$ is one of the $k$-nearest neighbors of $X_i$.

In the next theorem,
we establish the following CLT for the spectral measure of kNN.

\begin{theorem} \label{theorem11}
Let $k$ be any positive integer. Let $A_n$ be the adjacency matrix of  \textup{kNN} with a vertex set $\PP_n$. Then for any polynomial $f$,   there exists $\sigma_f^2 \ge 0$ such that  as $n \rightarrow \infty$, \begin{align}
  \frac{\textup{Tr} [f(A_n)] - \E \textup{Tr} [f(A_n)] }{\sqrt n}  \overset{\textup{d}}{\rightarrow} \mc N (0,\sigma_f^2).
\end{align}
 
\end{theorem}

\begin{proof}

It suffices to verify the conditions \eqref{eq:stabex} and \eqref{eq:stabexm} for any polynomial $f$.  To simplify the exposition, we present the argument for the case $d = 2$ and consider the monomial case $f(x) = x^m$. 

\medskip 

\noindent {\bf Stabilization.} 
Let $T_j(\ell)$, $1 \le j \le 6$, denote six disjoint equilateral circular sectors of radius $\ell$ sharing a common vertex at the origin, with their boundaries (and the vertex) excluded.  
We specify the orientation by requiring that $T_1(\ell)$ has one side parallel to the horizontal axis.  
The purpose of introducing these sectors is to build a geometric  shield that protects the node from external influence in the corresponding cone directions.

 For $\mc Q \in \textbf N$ and $\mathsf x\in \mc Q$,
\begin{align}
	\label{eq:rexp1}
\Rex^{(1)}(\mathsf x; \mc Q) := \max_{1\le j \le 6}\min\{\ell \in \mathbb N: \mc Q(\mathsf x + T_j(\ell)) \ge k + 1\}.
\end{align}  
Then for any $\mathsf x\in \mc Q$, 
\begin{align} \label{611}
    \text{all neighbors of $\mathsf x$ {(in $\mc Q$)} are contained in $B(\mathsf x,\Rex^{(1)}(\mathsf x;\mc Q))$.}
\end{align} 
As $\PP(T_i(\ell))$ is Poisson distributed with parameter $|T_i(\ell)| \in \Theta(\ell^2)$, there exists  $C>0$ depending on $k$ such that for all $\ell \in \mathbb N,$
\begin{align}
	  \label{eq:rexp5}	
    \P\big(\Rex^{(1)}(x; \PP \cup \{x\}) \ge \ell\big) &\le C(\ell^2)^{k}e^{- \ell^2},\qquad  \forall  x\in \mathbb R^2.
\end{align} 
Now,  we construct a stabilization radius $\Rex^{(2)}$ with respect to   $f(x) = x^m$:  
\begin{align}
	\label{eq:rexp}
	\Rex^{(2)}  := \min\big\{\ell \ge 100m^2: \Rex^{(1)}(0; \mc P \cup \{0\})  \vee \max_{\mathsf x \in \mc P \cap B(0, \ell)}\Rex^{(1)}(\mathsf  x; \mc P) \le \sqrt \ell\big\}.
\end{align}
Note that
$\Rex^{(2)}$ depends on the polynomial degree $m$, but we omit this dependence in the notation for readability.
By this definition,
\begin{align} \label{bb}
   \sqrt {\Rex^{(2)} } \ge \max_{\mathsf x \in \mc P \cap B(0, \Rex^{(2)})} \Rex^{(1)}(\mathsf  x; \mc P)  .
\end{align}
For any $\ell \ge 100m^2$,  by a union bound,
\begin{align}\label{614}
\P\big( \Rex^{(2)}   \ge \ell\big)& \le \P\big(\max_{\mathsf x \in \PP\cap B(0, \ell)}\Rex^{(1)}(\mathsf x; \mc P) \ge \sqrt \ell \big) + \P( \Rex^{(1)}(0; \mc P \cup \{0\})    \ge \sqrt \ell  )  \nonumber  \\
&\overset{\eqref{eq:rexp5}}{\le}  \E\big[\#\big\{\mathsf x \in \PP\cap B(0,  \ell ): \Rex^{(1)}(\mathsf x; \mc P ) \ge \sqrt \ell\big\} \big]  + C \ell^k e^{-\ell} \nonumber  \\
&= \int_{B(0, \ell )} \P\big(\Rex^{(1)}(x; \mc P \cup \{x\} ) \ge \sqrt \ell\big) \d x + C \ell^k e^{-\ell} \overset{\eqref{eq:rexp5}}{\le}   C\ell^{2+k}e^{-\ell},
\end{align}
where we used Mecke formula (see Lemma \ref{mecke}) in the second last equality.

Now, we verify the stabilizing property \eqref{eq:stabex} with a   stabilization radius $\Rex := \Rex^{(2)}$. Observe that in $k$-nearest neighbor
graphs, the addition of a new point cannot create any new edge that did not already exist before the addition, i.e.  for any $\mathsf x, \mathsf y\in\mc  Q \subset \mc  Q'$, we have $\mathsf y \notin \mathsf N(\mathsf x, \mc Q) \Rightarrow \mathsf y \notin \mathsf N(\mathsf x, \mc Q')$. {On the other hand, existing edges may be removed after the addition.} Therefore, the difference $D_f (\mc Q)$   can be written as $H_f^{(1)}(\mc Q)-H_f^{(2)}(\mc Q),$ where   
\begin{itemize}
    \item $H_f^{(1)}(\mc Q)$  denotes the contribution from the $m$-paths in $\mc Q \cup \{0\}$ passing through  the origin $0$;
    \item $H_f^{(2)}(\mc Q)$ denotes the contribution from the 
$m$-paths in $\mc Q$ that exist before the addition of the origin but are removed once the origin is inserted. 
\end{itemize}
We claim that
for any  $\mc A \in \tN $ with $\mc A \su \R^d \sm B\big(0, \Rex^{(2)}\big)$, two point processes
 $	 (\PP \cap B(0, \Rex^{(2)}))\cup \mc A  $ and $   \PP \cap B(0, \Rex^{(2)})  $ have the same values of $H_f^{(1)}$ and $H_f^{(2)}$.
This  would follow from    the following observations:
\been

\im  {for any $\mathsf x\in \PP \cap B(0, {\Rex^{(2)}}/2)$, all neighbors of $\mathsf x$ in $(\PP \cap B(0, \Rex^{(2)}))\cup \mc A $  are contained in $B\big(\mathsf x,\sqrt {{\Rex^{(2)}} } \big)$}. This in particular holds when $\mc A = \emptyset$;
\im  {for any $\mathsf x\in  (\PP \cap B(0, {\Rex^{(2)}}/2) ) \cup \{0\}$, all neighbors of  $\mathsf x$ in $(\PP \cap B(0, \Rex^{(2)}))\cup \mc A \cup \{0\}$  are contained in $B\big(\mathsf x,\sqrt {{\Rex^{(2)}} } \big)$}. This in particular holds when $\mc A = \emptyset$; 
\im  all $m$-paths in $  ( \PP \cap B(0, \Rex^{(2)}) ) \cup \mc A \cup \{0\} $, passing through the origin, lie inside $B\big(0, \Rex^{(2)}/4\big)$. Consequently, such a path is unaffected by the presence of $\mathcal A$;
\im  all $m$-paths in $(\mathcal{P} \cap B(0, \Rex^{(2)})) \cup \mathcal A$ that start in $B\big(0, \Rex^{(2)}/4\big)$ remain contained in $B\big(0, \Rex^{(2)}/2\big)$.  
Consequently, whether such a path disappears upon the removal of the origin is independent of the presence of $\mathcal A$;
 
\im  all $m$-paths in $(\PP \cap B(0, \Rex^{(2)}))\cup \mc A$ that start outside $ B\big(0, \Rex^{(2)}/4\big) $  are unaffected by the addition of  the origin 0. This in particular holds when $\mc A = \emptyset$.
\enen

For the two point configurations 
\((\mathcal{P} \cap B(0,\Rex^{(2)})) \cup \mathcal{A}\) and \(\mathcal{P} \cap B(0,\Rex^{(2)})\),  
the equality for \(H_f^{(1)}\) follows from (iii), while the equality for \(H_f^{(2)}\) follows from (iv) and (v).


~

We now prove the above properties. 
 For the property (i),  first note that for $\mathsf  x \in  \PP \cap B(0,\Rex^{(2)}/2)$, {we have $\Rex^{(1)}(\mathsf x;\PP) = \Rex^{(1)}(\mathsf x;\mc Q)$ for any $\mc Q \in \mathbf N$ satisfying $\mc Q \cap B(0, \Rex^{(2)}) = \PP \cap B(0, \Rex^{(2)})$. This follows from \eqref{611} along with the fact $\Rex^{(1)}(\mathsf x;\PP) < {\Rex^{(2)}}/2$ (see \eqref{bb}).} Thus by \eqref{611}  and  \eqref{bb}, all neighbors of $\mathsf  x$ in $\PP$ are contained in $B(\mathsf x,\sqrt{\Rex^{(2)}})$. As  $|\mathsf x| + \sqrt{\Rex^{(2)}} < \Rex^{(2)}$, by the connection rule of kNN, the same containment property holds  for the underlying point process $(\PP \cap B(0, \Rex^{(2)}))\cup \mc A$ as well.
 The property (ii) follows by the same reasoning.

The containment properties in $B\big(0,\Rex^{(2)}/2\big)$ stated in (iii) and (iv) follow  from (i) together with the inequality
\(m \sqrt{\Rex^{(2)}} < \Rex^{(2)}/4\) (recall that \(\Rex^{(2)} \ge 100m^2\); see \eqref{eq:rexp}). The claim that such an $m$-path is unaffected by the presence of $\mathcal A$ follows from the observation that its distance from $\mathcal A$ is greater than $\Rex^{(2)}/2> \sqrt{\Rex^{(2)}}$, together with (i).

  Finally, to verify (v), note that since $\Rex^{(2)}/4 > 2m\sqrt{\Rex^{(2)}}$, any such $m$-path must lie outside  
$B\big(0, \sqrt{\Rex^{(2)}}\big)$.  
Indeed, if such an $m$-path were to intersect $B\big(0, \sqrt{\Rex^{(2)}}\big)$, then one of its edges, with an endpoint in  
$B\big(0, \Rex^{(2)}/2\big)$, would necessarily have length at least
$ \frac{1}{m} ( \Rex^{(2)}/4 - \sqrt{\Rex^{(2)}} ) > \sqrt{\Rex^{(2)}}$,
contradicting (i).  
Together with (ii), this shows that such a path cannot be removed when the origin is added.

  ~

Therefore, we deduce that $\Rex := \Rex^{(2)}$ satisfies \eqref{eq:stabex} and thus is a stabilization radius. Since $ \Rex^{(2)}<\infty$ almost surely (see \eqref{614}), we finish the proof of \eqref{eq:stabex}.

~

\noindent {\bf Moment bound.} Next, we establish the moment bound  \eqref{eq:stabexm}.  
The argument follows the same general strategy as in the stabilization analysis, but with additional care required to handle boundary effects. For $\mc Q \in \mathbf N,$
as the graph is constructed only on the vertex set $\mc Q\cap W$, points outside $W$
are irrelevant for determining the $k$-nearest neighbors of a vertex $x\in W$.
Accordingly, when exploring a sector $x+T_j(\ell)$, we only need to enlarge it until its
intersection with $W$ stabilizes, i.e., $T_j(\ell)\cap W = T_j(\ell+1)\cap W$; beyond this
scale, further growth occurs entirely outside $W$ and cannot reveal any additional candidate
neighbors. This motivates the following definition, which stops either once $k+1$
points have been found or once the sector no longer uncovers new portions of the observation
window: for $\mathsf x \in \mc Q \cap W,$
\begin{align}
	\label{eq:rexp2}
	\Rex^{(1,W)}(\mathsf x; \mc Q) := \max_{1\le j \le 6}\min\Big\{\ell \in \mathbb N: \mc Q(\mathsf  x + T_j(\ell)) \ge k + 1 \text{ or } (\mathsf  x + T_j(\ell) ) \cap W = (\mathsf  x + T_j(\ell+1) )  \cap W\Big\}.
\end{align}
Then for $\mathsf x \in \mc Q \cap W,$
\begin{align} \label{6111}
    \text{all neighbors of $\mathsf x$ {(in $\mc Q\cap W$)} are contained in $B(\mathsf x,\Rex^{(1,W)}(\mathsf x;\mc Q))$.}
\end{align} 
 Similarly as in \eqref{eq:rexp},  for any cube $W$ with $0\in W,$ define
\begin{align}
	\label{eq:rexp3}
	\Rex^{(2,W)} := \min\big\{\ell \ge 100m^2: \Rex^{(1,W)}(0;  \PP \cup \{0\})  \vee \max_{\mathsf x \in  \PP \cap W \cap B(0, \ell)}\Rex^{(1,W)}(\mathsf  x; \PP) \le \sqrt \ell\big\}.
\end{align}
Note that  as  $\Rex^{(1,W)} (\mathsf x; \mc Q) \le \Rex^{(1)} (\mathsf x; \mc Q)  $,   we have 
  $\Rex^{(2,W)} \le \Rex^{(2)}  $. Thus $\Rex^{(2,W)}$ satisfies the similar upper bound for the right tail as in \eqref{614}: There exists $C>0$ (independent of $W$) such that for any cube $W$ with $0\in W$ and $\ell \ge 100m^2$,
\begin{align}\label{61444}
\P\big( \Rex^{(2,W)}   \ge  \ell\big) \le   C\ell^{2+k}e^{-\ell}.
\end{align}
In addition, analog of the aforementioned properties (i) and (ii) hold as well for the graph induced by $\mc P \cap W$: 
\been
\im  {for any $\mathsf x\in \PP \cap W \cap B(0, {\Rex^{(2,W)}}/2)$, all neighbors of $\mathsf x$ in $\PP \cap W  \cap B(0, \Rex^{(2,W)}) $  are contained in $B\big(\mathsf x,\sqrt {{\Rex^{(2,W)}} } \big)$}.  
\im  {for any $\mathsf x\in  (\PP \cap W  \cap B(0, {\Rex^{(2,W)}}/2) ) \cup \{0\}$, all neighbors of  $\mathsf x$ in $(\PP \cap W  \cap B(0, \Rex^{(2,W)})) \cup \{0\}$  are contained in $B\big(\mathsf x,\sqrt {{\Rex^{(2,W)}} } \big)$}.  
\enen  
Recall that $D_f(\PP \cap W)$   is  written as $H_f^{(1)}(\PP \cap W)-H_f^{(2)}(\PP \cap W)$. We compute the contributions of $m$-loops in $H_f^{(1)}(\PP \cap W) $  and $ H_f^{(2)}(\PP \cap W)$. It is clear that, by property (ii), all $m$-paths in $( \mc P \cap W) \cup \{0\} $ passing through the origin  are contained in \(B(0, m\sqrt{	\Rex^{(2,W)}})\). In addition, any $m$-path in $\mc P\cap W$ that disappears after adding the origin must intersect $B(0,\sqrt{\Rex^{(2,W)}})$ (otherwise it would remain unaffected by the property (ii)).
Once a path intersects $B(0,\sqrt{\Rex^{(2,W)}})$, each further step can move by at most $\sqrt{\Rex^{(2,W)}}$, due to the property (i) along with the fact $ m \sqrt{\Rex^{(2,W)}}< \Rex^{(2,W)}/2 $. This implies that  the entire $m$-path is contained in
$ B \big(0,(m+1)\sqrt{\Rex^{(2,W)}}\big)$.

Hence, we deduce that
\begin{align} \label{123}
    |D_f(\PP\cap W)|\le \PP\big(B(0, (m+1) \sqrt{\Rex^{(2,W)}})\big)^{m + 1}.
\end{align}
Therefore, using Hölder's inequality,
\begin{align*}
\E\big[|D_f(\PP\cap W)|^4\big] &=  \sum_{\ell=1}^\ff \E\big[|D_f(\PP\cap W)|^4 ; \Rex^{(2,W)}=\ell \big]  \\
 &\overset{\eqref{123}}{\le}   \sum_{\ell=1}^\ff \E\big[\PP\big(B(0, (m+1) \sqrt{ \ell})\big)^{4(m + 1)} ; \Rex^{(2,W)}=\ell \big]\\
    &\le \sum_{\ell=1}^\ff  \sqrt { \E \big[ \PP\big(B(0, (m+1) \sqrt{ \ell})\big)^{8(m + 1)} \big]}\sqrt { \P (  \Rex^{(2,W)}=\ell )  }  \\
&\le C\sum_{\ell=1}^\ff \ell^{8(m+1)} \sqrt{\P(\Rex^{(2,W)} = \ell)} \overset{\eqref{61444}}{<} \ff,
\end{align*} 
establishing the moment bound \eqref{eq:stabexm}.

\end{proof}

\noindent\emph{Variance positivity.}    We now show that $\s_f^2 > 0$ holds for polynomials with non-negative coefficients unless $f\equiv 0$, i.e. $f(x) = a_{m_1}x^{m_1} + a_{m_2}x^{m_2} +  \cdots + a_{m_u}x^{m_u}  $ ($m_1 > m_2 > \dots > m_u  \ge 0$) with $\min_{1\le i \le u}a_{m_i}>0 $.   By \cite[Theorem 1.1]{trinh2}, it suffices to show that $D_f(\PP)$ is non-degenerate. For the sake of simplicity, we consider the case $k = 1$, i.e., Poisson-nearest neighbor graph. Since there are no isolated nodes, if the addition of the origin does not remove any existing edges, then  for any $n\in \mathbb N$, 
$$L(f)((\PP \cup \{0\}) \cap W_n) \ge L(f)(\PP \cap W_n) +\min_{1\le i \le u}a_{m_i}.$$
This is because, after adding the origin $0$, a new $m_i$-hop path emanating from $0$ is created for every $1\le i\le u$, as there are no isolated nodes.
Therefore, it suffices to show that with a positive probability, adding the origin does not remove any edges.

Define the event $\mc E$ as follows:
\begin{align*}
    \mc E 
:= & \ \{\text{$B(0,1)$ contains no Poisson points}\}  \\
&\cap   \{\text{every point on $\partial B(0,1)$ lies within distance at most $1/4$ 
of some point in $\PP$}\}.
\end{align*}
 We claim that  on the event $\mc E$, the addition of the origin does not remove any edges. 
To see this, consider two points $X_i, X_j \in \mc P_n$ forming an edge, and suppose that the addition of $0$ removes this edge. This would imply, without loss of generality, that $|X_i| < |X_i - X_j|$. 

Under the event $\mc  E$, both $X_i$ and $X_j$ lie outside $B(0,1)$. Let  {$M := X_i/|X_i| \in \partial B(0,1)$} and $X_\ell$ be a Poisson point satisfying $|X_\ell - M| \le 1/4$, 
whose existence is guaranteed on the event $\mc  E$.
 Then,
$$|X_i - X_\ell| \le |X_i - M| + |M - X_\ell| \le (|X_i| - 1) + 1/4 < |X_i| < |X_i - X_j|  ,$$
contradicting the assumption that
$X_j$ was the closest neighbor to $X_i$ before the addition of $0$ (recall that we are considering the   case $k=1$).

Therefore it reduces to show that the event $\mc E$ defined above has positive probability. Take a finite set
of points $v_1,\dots,v_M \in \mathbb S^{d-1}$ such that $\mathbb S^{d-1} \subset \cup_{i=1}^M B\big(v_i, 1/8\big),$
and moreover the points can be chosen so that the balls 
$B(v_i,1/16)$ are pairwise disjoint.
Setting  $A_i := B(v_i,1/16) \cap (B(0,1))^c
$, we have that  $B(0,1)$ and the sets $A_1,\dots,A_M$ are mutually disjoint. Then    define the event 
\[
\mc E' := \Big\{\PP\big(B(0,1)\big) = 0\Big\} \cap 
       \bigcap_{i=1}^M \Big\{ \PP(A_i) \ge 1 \Big\}.
\]
By the mutual disjointness of these sets,
$
\mathbb P(\mc  E')
 = \mathbb P\big(\PP(B(0,1)) = 0\big) 
   \prod_{i=1}^M \mathbb P\big(\PP(A_i) \ge 1\big)> 0.
   $

We now claim that $\mc E' \subset \mc E$. Indeed, on $\mc E'$ we clearly have
$\PP(B(0,1)) = 0$. Moreover, for each $i$ there exists at least one Poisson
point $z_i \in A_i \subset B(v_i,1/16)$ with $|z_i| \ge 1$.
Given any point $y \in \mathbb S^{d-1}$, by the covering property there exists 
$i$ such that $|y - v_i| \le 1/8$. Then
\[
|y - z_i| \le |y - v_i| + |v_i - z_i| \le \frac{1}{8} + \frac{1}{16}
<\frac{1}{4}.
\]
Therefore $\mc E'$ 
implies $\mc E$, and hence  the event $\mc E$ has positive probability as well.

{Finally, we note that a similar argument applies for general $k \ge 2$. 
In this case, one considers the event that every point on the unit sphere 
lies within distance at most $1/4$ of $k$ distinct Poisson points.} 

%
%
\subsection{Relative neighborhood graphs}
\label{sss:rng}
We consider the Poisson relative neighborhood graph (RNG) (we refer to \cite{rng} for more details). For $x,z\in \mathbb R$, define  $B_{x,z}:=B(x,|x-z|) \cap B(z,|x-z|)$. For $\mc Q \in \mathbf N,$ two points $X_i,X_j \in \mc Q$ are connected by an edge if and only if
$$
\mc Q \cap B_{X_i,X_j} = \emptyset.
$$

In the next theorem,
we establish the following CLT for the spectral measure of RNG.

\begin{theorem} \label{theorem12}
Let $A_n$ be the adjacency matrix of \textup{RNG} with a vertex set $\PP_n$. Then for any polynomial  $f$,   there exists $\sigma_f^2 \ge 0$ such that  as $n \rightarrow \infty$, \begin{align}
  \frac{\textup{Tr} [f(A_n)] - \E \textup{Tr} [f(A_n)] }{\sqrt n}  \overset{\textup{d}}{\rightarrow} \mc N (0,\sigma_f^2).
\end{align}
 
\end{theorem}

\begin{proof}

As before, it suffices to verify the conditions \eqref{eq:stabex} and \eqref{eq:stabexm}.

\medskip

\noindent {\bf Stabilization.} 
We define the stabilization radius similarly to the kNN
graph.  To ease the presentation, we restrict ourselves to the case $d=2$ as before. We let $T_j(\ell)$, $1\le j \le 13$, be disjoint equilateral circular sectors of radius $\ell$ sharing one vertex at the origin. Note that a similar splitting into sectors has been used to prove stabilization in \cite[Section 6]{yukCLT}. However, due to the fact that we consider the RNG (instead of the $k$-nearest neighbour graph), we need to split the plane into a larger number of sectors. Then for $\mc Q \in \textbf N$ and $\mathsf x\in \mc Q$, define
\begin{align}
\label{eq:rexp1_rng}
\Rex^{(1)}(\mathsf x; \mc Q) := \max_{1\le j\le 13}\min\{\ell \in \mathbb N: \mc Q(\mathsf x + T_j(\ell ))\ge 1\}.
\end{align}
By the same reasoning as in kNN,  $\Rex^{(1)}$ satisfies the bound \eqref{eq:rexp5} with $k=1$. Next,
set  
\begin{align}
	\label{eq:rexp_rng}
	\Rex^{(2)}  := \min\Big\{\ell  \ge 100m^2:  \Rex^{(1)}(0; \mc P \cup \{0\})  \vee  \max_{\mathsf x \in \PP\cap B(0, \ell)}\Rex^{(1)}(\mathsf x; \PP) \le \sqrt \ell\Big\}.
\end{align}
Then, the points (i)--(v) stated in the  kNN
graph case follow analogously. To show (i), for the sake of simplicity we assume $\mathsf x =0$, and we need to check that for $\mathsf y \in \PP$ with $\PP \cap B_{0,\mathsf  y }=\es$, it holds that $|\mathsf y| \le \sqrt {\Rex^{(2)}}$. To see this, let $C(\mathsf y)$ be the circular sector with apex 0, radius $|\mathsf  y|$ and angle $2\pi/3$, whose axis passes through $\mathsf  y$. Note that $C(\mathsf  y) \subset B_{0,\mathsf  y}$ (see Figure \ref{fig:rng}). Since the cones $T_j(|\mathsf  y|)$ have an opening angle $2\pi/13$, we have $T_j(|\mathsf  y|) \subset  C(\mathsf  y) \subset B_{0,\mathsf  y}$ for some $1\le j\le 13$. As there are no Poisson points in $B_{0,\mathsf  y}$ and thus in $T_j(|\mathsf  y|),$ this implies   ${\Rex^{(1)}} (0; \mc Q) \ge |\mathsf  y|$ and thus $|\mathsf  y| \le \sqrt {\Rex^{(2)}}$.

\begin{figure}
\begin{tikzpicture}[scale=0.8]
    \coordinate (y) at (2,1.5);
    \coordinate (O) at (2,-1.5);

    \draw[name path=diskO] (O) circle (3.0);
    \draw[name path=disky] (y) circle (3.0);

    \path [name intersections={of=diskO and disky, by={A,B}}];

    \draw[blue, ultra thick] (O) -- (A);
    \draw[blue, ultra thick] (O) -- (B);

    \filldraw (y) circle (2pt) node[above right] {\Large $y$};
    \filldraw (O) circle (2pt) node[below right] {\Large $0$};
    \filldraw (A) circle (2pt) node[above left] {};
    \filldraw (B) circle (2pt) node[below left] {};
    \draw[dashed] (O)--(y);

    \draw pic["$2\pi/3$", draw=black,-, angle radius=1cm]
        {angle=A--O--B};

\end{tikzpicture}
\caption{Illustration of the stabilization radius in the RNG}
\label{fig:rng}
\end{figure}
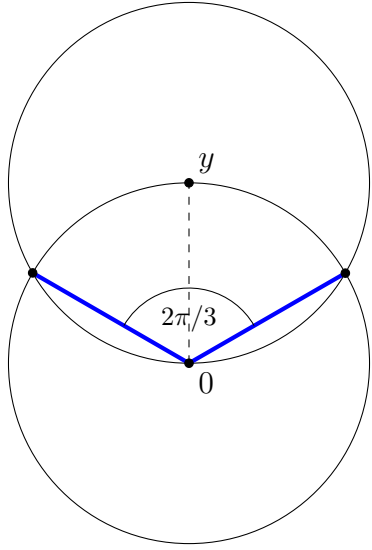

\medskip

\noindent {\bf Moment bound.} The moment bound follows exactly as in the kNN
graph.
\end{proof}

	\noindent\emph{Variance positivity.} As in the case of the nearest-neighbor graph discussed above, it suffices to show that, with positive probability, adding the origin does not remove any existing edges. Consider the event $\mc E$ that $\PP \cap B(0,1)=\es$ and $\PP(T_i(1+\eps)) \ge 1$ for all $i \in I = \{1,2,\dots,13\}$, where $\eps = 10^{-5}.$   This event occurs with positive probability. We verify that under this event $\mc E,$ adding the origin cannot remove any existing edges.

    To see this, consider two points $X_i, X_j \in \mc P_n$ forming an edge, and suppose that the addition of $0$ removes this edge. This means that  $0 \in B_{X_i,X_j}$. Recalling $\PP(T_\ell(1+\eps)) \ge 1$ for all $\ell\in I$ under the event $\mc E$, we obtain the contradiction once we prove the following claim.
    
 \textbf{Claim.} Under the event $\mc E$, we have $T_\ell(1+ \eps) \subseteq B_{X_i,X_j}$  for some $\ell \in I$.

 
To prove this claim, we fix two points $x,y \in \partial B_{X_i,X_j}$ with $|x-(X_i+X_j)/2| = (1+\eps)|X_i-X_j|/2$ and $|y-(X_i+X_j)/2| = (1+\eps)|X_i-X_j|/2$ such that $\langle X_i-x,x\rangle < 0 $ and $\langle X_j-y,y\rangle < 0 $ (see Figure \ref{fig:rng2}). Then the claim follows if we can show that (i) $\max\{|x|,|y|\} \ge 1+\eps$ and (ii) $\alpha:=\angle(y, 0, x) \ge \frac{4\pi}{13}$. Indeed, since the $T_\ell(1+\epsilon)$ have angular radius ${\pi}/{13}$,  the triangle $(y,0,x)$ contains some $T_\ell(1+\epsilon)$.
(i) Without loss of generality, assume that $|X_i| \le |X_j|$. Since the point $(X_i+X_j)/2$ minimizes the ratio $|x-z|/|X_i-z|$ among all points $z \in B_{X_i,X_j}$ with $|X_i-z| \le |X_j - z|$ and $\langle X_i-x ,x-z\rangle <0$, we have
$$
|x| = |X_i| \frac{|x|}{|X_i|} \ge |X_i| \frac{|x-(X_i+X_j)/2|}{|X_i-(X_i+X_j)/2|} = |X_i|(1+\eps)  \ge 1+\eps
$$
and, analogously, also $|y| \ge 1+\eps$.

\usetikzlibrary{intersections, calc, angles, quotes}
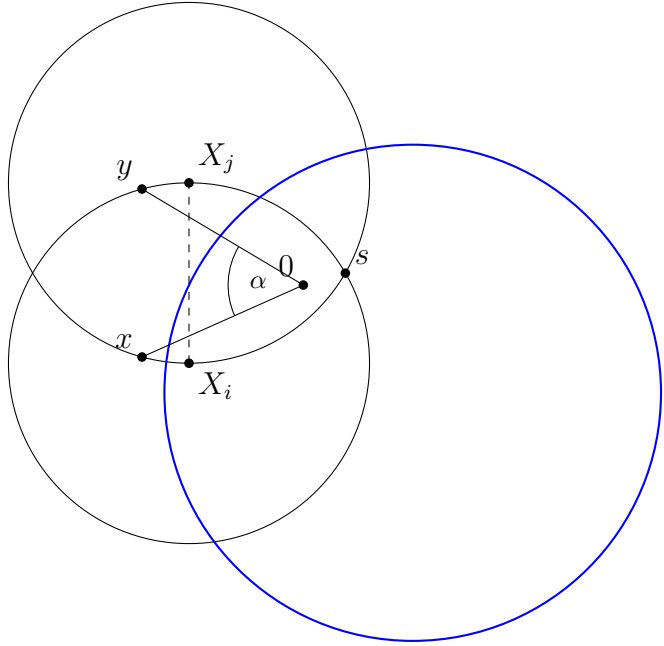
\begin{figure}
\begin{tikzpicture}[scale=0.8]
    \coordinate (y) at (2,1.5);
    \coordinate (O) at (2,-1.5);
    \coordinate (z) at (1.2193,-1.3967);
    \coordinate (z2) at (1.2193,1.3967);
    \coordinate (w) at (3.9,-0.2);

     \path[name path=diskO] (O) circle (3);     
    \path[name path=disky] (y) circle (3); 

    \path [name intersections={of=diskO and disky, by={A,B}}];

    \draw[name path=diskO] (O) circle (3.0);
    \draw[name path=disky] (y) circle (3.0);

    \filldraw (y) circle (2pt) node[above right] {\Large $X_j$};
    \filldraw (O) circle (2pt) node[below right] {\Large $X_i$};
    \filldraw (z) circle (2pt) node[above left] {\Large $x$};
    \filldraw (z2) circle (2pt) node[above left] {\Large $y$};
    \filldraw (w) circle (2pt) node[above left] {\Large $0$};
    \filldraw (A) circle (2pt) node[above right] {\Large $s$};

\draw[blue, thick] (5.718,-1.992) circle (4.126);

    \draw pic["$\alpha$", draw=black,-, angle radius=1cm]
        {angle=z2--w--z};

    \draw[] (z2)--(w);
    \draw[] (z)--(w);
    \draw[dashed] (O)--(y);

\end{tikzpicture}
\caption{The Apollonius circle of all points $z$ with equal ratio $\frac{|x-z|}{|X_i-z|}=\frac{|x-(X_i+X_j)/2|}{|X_i-(X_i+X_j)/2|}$ is drawn in blue. In particular, this ratio is larger for all points inside this circle.}
\label{fig:rng2}
\end{figure}

(ii) Let $s$ be the intersection point of $\partial B(X_i,|X_i-X_j|)$ and $\partial B(X_j,|X_i-X_j|)$ with $\langle X_i,s\rangle <0$. It follows from the inscribed angle theorem that the point $s$ minimizes the angle $\angle (y,z,s)$ among all points $z \in B_{X_i,X_j}$ with $\langle X_i-x,x-z\rangle <0$. Thus,
$$
\alpha=\angle (y,0,x) \ge \angle (y,s,x).
$$
Recalling $\eps = 10^{-5}$, we have $\angle(y, s, x) \ge \frac{4\pi}{13}$, which proves the claim.

%
\subsection{Outlook of possible further examples}
\label{ssec:out}

\subsubsection{Scale-Free Gilbert graphs}
As described above, the classical Gilbert graph (or RGG) connects pairs of points that lie within a fixed distance of each other.  
A scale-free Gilbert graph extends this model by assigning to each vertex a \emph{random} connection radius $R_i$, typically drawn from a heavy-tailed distribution \cite{braz}.  
Vertices $X_i$ and $X_j$ are connected if and only if $|X_i - X_j| \le R_i + R_j$.  
The resulting network is simultaneously geometric and heavy-tailed, capturing—for instance—communication systems in which transmission ranges or influence domains vary widely across nodes.

The heavy-tailed radii lead to behavior that can differ dramatically from the RGG.  
In particular, unless the tail index of the radius distribution is sufficiently large, the moment bounds used in the proof of Theorem~\ref{thm:polt1} do not hold.
Furthermore, when the degree distribution lacks a finite second moment, we conjecture that the limiting fluctuation of the spectral measure is no longer asymptotically normal but instead converge to an $\alpha$-stable law.

\subsubsection{Hyperbolic random graphs}Hyperbolic random graphs place vertices in a hyperbolic space, typically $\mathbb{H}^2$, and connect them with probabilities that decay with hyperbolic distance.  
A common construction samples points in the hyperbolic plane according to a density that produces power-law degree distributions \cite{gpp}.  
Each vertex is assigned a radial and an angular coordinate, where the radial component encodes a notion of “popularity’’ and the angular component represents “community.’’  
Edges are then formed with higher probability between vertices that are close in hyperbolic distance.  
The geometry of the hyperbolic metric induces an inherent hierarchical organization among nodes, making this framework a natural model for scale-free networks with strong clustering.

Hyperbolic random graphs share with the scale-free Gilbert model the feature of heavy-tailed degrees.  
Consequently, as in the scale-free Gilbert case, the moment bounds used in the proof of Theorem~\ref{thm:polt1} fail in this setting.  
When the degree distribution has infinite second moment, we likewise conjecture that the limiting fluctuations are no longer Gaussian but instead converge to a stable law.

\subsubsection{Delaunay triangulation} 
A  {Delaunay triangulation} of a planar point set is a triangulation with the property that no point lies in the interior of the circumcircle of any triangle. 
Equivalently, it maximizes the minimum angle over all triangulations of the point set, and it is the planar dual of the Voronoi diagram. 
The Delaunay triangulation is unique whenever no four points are co-circular; see \cite{rng} for further background and properties.

As in the \(k\)-nearest neighbor case, it is somewhat difficult to locate an explicit verification of the stabilization and moment conditions \eqref{eq:stabex} and \eqref{eq:stabexm} in the existing literature. 
Nevertheless, these properties follow from straightforward adaptations of the arguments in \cite[Section~9]{yukCLT}. 
To avoid redundancy, we omit the details.

The verification of variance positivity, however, is substantially more delicate than in the \(k\)-nearest neighbor setting. 
The key issue is that inserting a single point into a Delaunay triangulation necessarily deletes some existing edges, so the monotonicity argument used in Section~\ref{sss:rng} no longer applies. 
While we expect that variance positivity still holds, establishing it appears to require a finer and more model-specific analysis.

\bibliographystyle{abbrv}
\bibliography{lit}

\end{document}